 \documentclass[12pt, reqno]{amsart}
 \usepackage[margin=0.80in]{geometry}

\usepackage{amsfonts}

\usepackage{setspace}
\usepackage{graphicx}
\usepackage{eufrak}
\usepackage{mathrsfs}
\usepackage{yfonts}
\usepackage{hyperref}
\usepackage{graphicx}
\usepackage{amsmath, amsthm, amscd, amsfonts, amssymb, graphicx, color}
\usepackage{url}
\usepackage{enumerate}
\makeatletter
\let\reftagform@=\tagform@
\def\tagform@#1{\maketag@@@{(\ignorespaces\textcolor{blue}{#1}\unskip\@@italiccorr)}}
\renewcommand{\eqref}[1]{\textup{\reftagform@{\ref{#1}}}}
\makeatother
\usepackage{hyperref}
\hypersetup{colorlinks=true, linkcolor=red, anchorcolor=green,
    citecolor=cyan, urlcolor=red, filecolor=magenta, pdftoolbar=true}

\usepackage{epsfig}        
\usepackage{epic,eepic}       

\newtheorem{theorem}{Theorem}
\theoremstyle{plain}

\newtheorem{cor}{Corollary}

\newtheorem{thm}{Theorem}
\newtheorem{rem}{Remark}

\newtheorem{remark}{Remark}

\numberwithin{equation}{section}

\begin{document}
   \title[Bounds for the  Difference between two \v{C}eby\v{s}ev
   functional]{Bounds for the difference between two \v{C}eby\v{s}ev
   	functionals}
   \author[M.W. Alomari]{Mohammad W. Alomari}
   \address{ Department of Mathematics, Faculty of Science and
   	Information Technology, Irbid National University, P.O. Box 2600,
   	Irbid, P.C. 21110, Jordan.} \email{mwomath@gmail.com}

   \subjclass[2000]{26D15} \keywords{\v{C}eby\v{s}ev functional,
   	Gr\"{u}ss inequality}
   
   \begin{abstract}
   	In this work,  a  generalization of pre-Gr\"{u}ss inequality is
   	established. Several bounds for the difference between two
   	\v{C}eby\v{s}ev functional are proved.
   \end{abstract}

    \maketitle
    \section{Introduction}

It is well known that for a continuous function $f$ defined on
$[a,b]$, the integral mean-value theorem (IMVT)  guarantees $x\in
[a,b]$ such that
\begin{align}
f\left({x}\right)=\frac{1}{b-a}\int_a^b{f\left({t}\right)dt}.\label{eq4.1.1}
\end{align}
On the other hand, for a monotonic function $g:[a,b]\to
\mathbb{R}$ that does not change sign in the interval $[a,b]$, the
weighted IMVT reads that there exists $x\in [a,b]$ such that
\begin{equation}
\int_a^b{f\left({t}\right)g\left({t}\right)dt}=f\left({x}\right)
\int_a^b{g\left({t}\right)dt}.\label{eq4.1.2}
\end{equation}
If one replaces the value of $f(x)$ in (\ref{eq4.1.2}) by its
value in (\ref{eq4.1.1}) then we get
\begin{equation}
\int_a^b{f\left({t}\right)g\left({t}\right)dt}=\frac{1}{b-a}\int_a^b{f\left({t}\right)dt}
\int_a^b{g\left({t}\right)dt}.\label{eq4.1.3}
\end{equation}
To get weighted values in  (\ref{eq4.1.3}) we divide the both
sides by the quantity `$b-a$'  to get
\begin{equation}
\frac{1}{b-a}\int_a^b{f\left({t}\right)g\left({t}\right)dt}=\frac{1}{b-a}\int_a^b{f\left({t}\right)dt}
\cdot \frac{1}{b-a}\int_a^b{g\left({t}\right)dt},\label{eq4.1.4}
\end{equation}
which means in such way that the weighted product of two functions
equal to the product of weights of that functions.

The difference between these weights
\begin{align}
\label{identity}\mathcal{T}_a^b\left( {f,g} \right) = \frac{1}{{b
		- a}}\int_a^b {f\left( t \right)g\left( t \right)dt}  -
\frac{1}{{b - a}}\int_a^b {f\left( t \right)dt}  \cdot \frac{1}{{b
		- a}}\int_a^b {g\left( t \right)dt}.
\end{align}
is called `the \v{C}eby\v{s}ev functional', which plays an
important role in Numerical Approximations and Operator Theory.
For more detailed history see \cite{MPF}.

The most famous bounds for the \v{C}eby\v{s}ev functional are
incorporated in the following theorem:
\begin{theorem}
	\label{thm1}Let $f,g:[a,b] \to \mathbb{R}$ be two absolutely
	continuous functions, then
	\begin{align}
	\left|{\mathcal{T}_a^b\left( {f,g} \right)} \right| \le\left\{
	\begin{array}{l} \frac{{\left( {b - a} \right)^2 }}{{12}}\left\|
	{f'} \right\|_\infty  \left\| {g'} \right\|_\infty
	,\,\,\,\,\,\,\,\,\,{\rm{if}}\,\,f',g' \in L_{\infty}[a,b],\,\,\,\,\,\,\,\,{\rm{proved \,\,in \,\,}}{\text{\cite{Cebysev}}}\\
	\\
	\frac{1}{4}\left( {M_1 - m_1} \right)\left( {M_2 - m_2}
	\right),\,\,\, {\rm{if}}\,\, m_1\le f \le M_1,\,\,\,m_2\le g \le
	M_2, \,\,{\rm{proved \,\,in \,\,}}{\text{\cite{Gruss}}}\\
	\\
	\frac{{\left( {b - a} \right)}}{{\pi ^2 }}\left\| {f'} \right\|_2
	\left\| {g'} \right\|_2
	,\,\,\,\,\,\,\,\,\,\,\,\,\,\,\,\,{\rm{if}}\,\,f',g' \in
	L_{2}[a,b],\,\,\,\,\,\,\,\,\,\,\,{\rm{proved \,\,in\,\,
	}}{\text{\cite{L}}}\\
	\\
	\frac{1}{8}\left( {b - a} \right)\left( {M - m} \right) \left\|
	{g'} \right\|_{\infty},\,\,\, {\rm{if}}\,\, m\le f \le M,\,g' \in
	L_{\infty}[a,b], \,\,{\rm{proved \,\,in \,\,}}{\text{\cite{O}}}
	\end{array} \right.
	\end{align}
	The constants $\frac{1}{12}$, $\frac{1}{4}$, $\frac{1}{\pi^2}$ and
	$\frac{1}{8}$ are the best possible.
\end{theorem}
Many authors were studied the functional (\ref{identity}) and
therefore various bounds have been implemented, for more new
results and generalizations the reader may refer to
\cite{alomari1},\cite{alomari2},\cite{Cerone1},\cite{Cerone2},\cite{Cerone4},\cite{Dragomir1},\cite{HH}
and \cite{P}.\\

In 2001,  Cerone  \cite{Cerone5} established the following
identity for the \v{C}eby\v{s}ev functional:
\begin{thm}
	\label{thm2}Let $f,g:[a,b]\to \mathbb{R}$ be such that $f$ is of
	bounded variation and $g$ is continuous on $[a,b]$. Then, we have
	the following representation:
	\begin{align}
	\label{eq4.2.12}\mathcal{T}_a^b\left( {f,g} \right) =
	\frac{1}{{\left( {b - a} \right)^2 }}\int_a^b {\left[ {\left( {t -
				a} \right)\int_t^b {g\left( s \right)ds}  - \left( {b - t}
			\right)\int_a^t {g\left( s \right)ds} } \right]df\left( t
		\right)}.
	\end{align}
\end{thm}

In 2007, Dragomir \cite{Dragomir2} established three equivalent
identities that generalized Cerone identity (\ref{eq4.2.12}) for
Riemann-Stieltjes integrals, in case of Riemann integral Dragomir
representation incorporated in the following theorem.
\begin{thm}
	\label{thm3}Let $f,g:[a,b]\to \mathbb{R}$ be such that $f$ is of
	bounded variation and $g$ is Lebesgue integrable on $[a,b]$. Then,
	\begin{align}
	\label{eq4.2.14}\mathcal{T}_a^b\left( {f,g} \right) =
	\frac{1}{{\left( {b - a} \right)^2
	}}\int_a^b {\left[ {\left( {t - a} \right)\int_a^b {g\left( t
				\right)dt}  - \left( {b - a} \right)\int_a^t {g\left( s \right)ds}
		} \right]df\left( t \right)}.
	\end{align}
\end{thm}

The absolute difference between two integral means was studied
firstly by Barnett et al. in \cite{Barnett} and then by Cerone and
Dragomir in \cite{Cerone3}, we may summarize the obtained results,
as follow:

$\bullet$ For an  absolutely continuous function $f$ defined on
$[a,b]$ and for all $a\leq c<d\leq b$, we have
\begin{align}
& \left\vert {\frac{1}{{b-a}}\int_{a}^{b}{f\left( t\right) dt}-\frac{1}{{d-c}%
	}\int_{c}^{d}{f\left( s\right) ds}}\right\vert  \label{Bar4.3.1} \\
& \leq \left[ {\frac{1}{4}+\left( {\frac{{\left( {a+b}\right) /2-\left( {c+d}%
				\right) /2}}{{\left( {b-a}\right) -\left( {d-c}\right) }}}\right) ^{2}}%
\right] \left[ {\left( {b-a}\right) -\left( {d-c}\right) }\right]
\left\Vert
{f^{\prime }}\right\Vert _{\infty }  \notag \\
& \leq \frac{1}{2}\left[ {\left( {b-a}\right) -\left( {d-c}\right)
}\right] \left\Vert {f^{\prime }}\right\Vert _{\infty }  \notag
\end{align}
and
\begin{align}
& \left\vert {\frac{1}{{b-a}}\int_{a}^{b}{f\left( t\right) dt}-\frac{1}{{d-c}%
	}\int_{c}^{d}{f\left( s\right) ds}}\right\vert  \label{Cer4.3.2} \\
& \leq \left\{
\begin{array}{l}
\frac{{\left( {b-a}\right) }}{{\left( {q+1}\right) ^{1/q}}}\left[ {1+\left( {%
		\frac{\rho }{{1-\rho }}}\right) ^{q}}\right] ^{1/q}\left[
{v^{q+1}+\lambda
	^{q+1}}\right] ^{1/q}\left\Vert {f^{\prime }}\right\Vert _{p}, \\
\,\,\,\,\,\,\,\,\,\,\,\,\,\,\,\,\,\,\,\,\,\,\,\,\,\,\,\,\,\,\,\,\,\,\,\,\,\,%
\,\,\,\,\,\,\,\,\,\,\,\,\,\,\,f^{\prime }\in L_{p}\left[
{a,b}\right]
,\,\,\,1\leq p<\infty ,\,\,\frac{1}{p}+\frac{1}{q}=1; \\
\\
\frac{1}{2}\left[ {1-\rho +\left\vert {v-\lambda }\right\vert
}\right] \left\Vert {f^{\prime }}\right\Vert
_{1},\,\,\,\,\,\,\,\,\,\,\,f^{\prime
}\in L_{1}\left[ {a,b}\right] ;%
\end{array}%
\right.  \notag
\end{align}%
where $\left( {b-a}\right) v=c-a$, $\left( {b-a}\right) \rho =d-c$ and $%
\left( {b-a}\right) \lambda =b-d$.

$\bullet$ For a H\"{o}lder continuous function $f$ of order
$r\in(0,1]$ with constant $H>0$ on $[a,b]$, we have
\begin{equation}
\left\vert {\frac{1}{{b-a}}\int_{a}^{b}{f\left( t\right) dt}-\frac{1}{{d-c}}%
	\int_{c}^{d}{f\left( s\right) ds}}\right\vert \leq H\frac{{\left( {c-a}%
		\right) ^{r+1}+\left( {b-d}\right) ^{r+1}}}{{\left( {r+1}\right) \left[ {%
			\left( {b-a}\right) -\left( {d-c}\right) }\right] }}.
\label{Cer4.3.3}
\end{equation}%

$\bullet$ For a function $f$ of bounded variation on $[a,b]$, we
have
\begin{align}
& \left\vert {\frac{1}{{b-a}}\int_{a}^{b}{f\left( t\right) dt}-\frac{1}{{d-c}%
	}\int_{c}^{d}{f\left( s\right) ds}}\right\vert  \label{Cer4.3.4} \\
& \leq \left\{
\begin{array}{l}
\left[ {\frac{{b-a-\left( {d-c}\right) }}{2}+\left\vert {\frac{{c+d}}{2}-%
		\frac{{a+b}}{2}}\right\vert }\right] \frac{{\bigvee_{a}^{b}\left( f\right) }%
}{{b-a}}; \\
\\
L\frac{{\left( {c-a}\right) ^{2}+\left( {b-d}\right)
		^{2}}}{{2\left[ {\left( {b-a}\right) -\left( {d-c}\right) }\right]
}};\,\,\,\,\,\,\text{if\thinspace
	\thinspace f\thinspace \thinspace is\thinspace \thinspace L-Lipschitzian} \\
\\
\left( {\frac{{b-d}}{{b-a}}}\right) f\left( b\right) -\left( {\frac{{c-a}}{{%
			b-a}}}\right) f\left( a\right) +\left[ {\frac{{c+d-\left( {a+b}\right) }}{{%
			b-a}}}\right] f\left( {s_{0}}\right) ;\,\text{\thinspace } \\
\text{if\thinspace \thinspace f\thinspace \thinspace is\thinspace
	\thinspace
	monotonic\thinspace \thinspace nondecreasing}%
\end{array}%
\right.  \notag
\end{align}%
where, $s_{0}=\frac{{cb-ad}}{{\left( {b-a}\right) -\left( {d-c}\right) }}\in %
\left[ c,d\right] $.

For recent results the reader may refer to \cite{alomari3}, where
the author used \eqref{eq4.2.14} to obtain several bounds for the
\v{C}eby\v{s}ev functional. Bounds for the difference between two
Stieltjes integral means was presented in
\cite{alomari4}.\\

Let $g :\left[ {\alpha ,\beta } \right] \longrightarrow
\mathbb{R}$ be any integrable function  and define $\Psi:\left[
{\alpha ,\beta } \right] \longrightarrow \mathbb{R}$, such that
\begin{align*}
\Psi_g\left( {t;\alpha ,\beta } \right): = \int_{\alpha}^t
{g\left( s \right)ds}  - \frac{{t - \alpha}}{{\beta -
		\alpha}}\int_{\alpha}^{\beta} {g\left( s \right)ds}.
\end{align*}
From \eqref{eq4.2.14}, it is easy to observe  the following
representation of the \v{C}eby\v{s}ev functional
\begin{align*}
\mathcal{T}_{\alpha}^\beta  \left( {f,g} \right): = -
\frac{1}{{\beta - \alpha}}\int_{\alpha}^{\beta} {\Psi _g \left(
	{t;\alpha,\beta} \right)df\left( t \right)}.
\end{align*}
In this work by utilizing  the inequalities
(\ref{Bar4.3.1})--(\ref{Cer4.3.4}), several new bounds for the
absolute \textit{Difference between two \v{C}eby\v{s}ev
	functional} $\mathcal{T}_{a}^{v} \left( {f,g} \right) -
\mathcal{T}_{u}^{b} \left( {f,g} \right)$, for all $a \le u < v
\le b$  are provided.

Let us start by providing the following refinements of
pre-Gr\"{u}ss inequality, which states that for any two integrable
mappings defined on $[a,b]$, the inequality
\begin{align}
\label{eq2.1} \mathcal{T}_{a}^b\left( {f,g} \right) \le \left[
{\mathcal{T}_{a}^b\left( {f,f} \right)} \right]^{1/2} \cdot\left[
{ \mathcal{T}_{a}^b\left( {g,g} \right)} \right]^{1/2},
\end{align}
holds and  sharp (see \cite{Gruss}). Trivially, by applying AM--GM
inequality on the right hand side of \eqref{eq2.1}, we get
\begin{align}
\left[ {\mathcal{T}_{a}^b\left( {f,f} \right)} \right]^{1/2}
\cdot\left[ { \mathcal{T}_{a}^b\left( {g,g} \right)}
\right]^{1/2}\le  \frac{\mathcal{T}_a^b \left( {f,f}
	\right)+\mathcal{T}_a^b \left( {g,g} \right)}{2}.
\end{align}
We may generalize the pre-Gr\"{u}ss inequality \eqref{eq2.1} as
follows:
\begin{thm}
	\label{thm4} Let $f,g:[a,b]\to \mathbb{R}$ be two integrable
	mappings, then
	\begin{align}
	\label{eq2.2}&\left| {\mathcal{T}_a^v \left( {f,g} \right) -
		\mathcal{T}_u^b \left( {f,g} \right)} \right|
	\nonumber\\
	&\le \left( {\mathcal{T}_a^v \left( {f,f} \right)} \right)^{1/2}
	\left( {\mathcal{T}_a^v \left( {g,g} \right)} \right)^{1/2}  +
	\left( {\mathcal{T}_u^b \left( {f,f} \right)} \right)^{1/2} \left(
	{\mathcal{T}_u^b \left( {g,g} \right)} \right)^{1/2}
	\\
	&\le \frac{1}{2}\left[\mathcal{T}_a^v \left( {f,f}
	\right)+\mathcal{T}_a^v \left( {g,g} \right)+\mathcal{T}_u^b
	\left( {f,f} \right)+\mathcal{T}_u^b \left( {g,g} \right)
	\right],\nonumber
	\end{align}
	for all $a\le u<v\le b$. The double inequality is sharp.
\end{thm}

\begin{proof}
	Simply using the \eqref{eq2.1}, we have
	\begin{align*}
	&\left| {\mathcal{T}_a^v \left( {f,g} \right) - \mathcal{T}_u^b
		\left( {f,g} \right)} \right|^2
	\\
	&\le \left( {\mathcal{T}_a^v \left( {f,g} \right)} \right)^2  +
	2\mathcal{T}_a^v \left( {f,g} \right) \cdot \mathcal{T}_u^b \left(
	{f,g} \right) +  \left( {\mathcal{T}_u^b \left( {f,g} \right)}
	\right)^2
	\\
	&\le \left( {\mathcal{T}_a^v \left( {f,f} \right)} \right)\left(
	{\mathcal{T}_a^v \left( {g,g} \right)} \right) + 2\mathcal{T}_a^v
	\left( {f,g} \right) \cdot \mathcal{T}_u^b \left( {f,g} \right) +
	\left( {\mathcal{T}_u^b \left( {f,f} \right)} \right)\left(
	{\mathcal{T}_u^b \left( {g,g} \right)} \right)
	\\
	&= \left[ {\left( {\mathcal{T}_a^v \left( {f,f} \right)}
		\right)^{1/2} \left( {\mathcal{T}_a^v \left( {g,g} \right)}
		\right)^{1/2}  + \left( {\mathcal{T}_u^b \left( {f,f} \right)}
		\right)^{1/2} \left( {\mathcal{T}_u^b \left( {g,g} \right)}
		\right)^{1/2} } \right]
	\\
	&\qquad\qquad\qquad\times \left( {\mathcal{T}_a^v \left( {f,f}
		\right)} \right)^{1/2} \left( {\mathcal{T}_a^v \left( {g,g}
		\right)} \right)^{1/2}
	\\
	&\qquad+ \left[ {\left( {\mathcal{T}_u^b \left( {f,f} \right)}
		\right)^{1/2} \left( {\mathcal{T}_u^b \left( {g,g} \right)}
		\right)^{1/2}  + \left( {\mathcal{T}_a^v \left( {f,f} \right)}
		\right)^{1/2} \left( {\mathcal{T}_a^v \left( {g,g} \right)}
		\right)^{1/2} } \right]
	\\
	&\qquad\qquad\qquad\times \left( {\mathcal{T}_u^b \left( {f,f}
		\right)} \right)^{1/2} \left( {\mathcal{T}_u^b \left( {g,g}
		\right)} \right)^{1/2}
	\\
	&= \left[ {\left( {\mathcal{T}_a^v \left( {f,f} \right)}
		\right)^{1/2} \left( {\mathcal{T}_a^v \left( {g,g} \right)}
		\right)^{1/2}  + \left( {\mathcal{T}_u^b \left( {f,f} \right)}
		\right)^{1/2} \left( {T_u^b \left( {g,g} \right)} \right)^{1/2} }
	\right]^2
	\end{align*}
	and this implies the first inequality in \eqref{eq2.2}. The second
	inequality follows by applying the AM--GM inequality.  The
	sharpness follows by letting $f=g=x$.
\end{proof}

\begin{remark}
	We note that \eqref{eq2.2} reduces to \eqref{eq2.1} by setting
	$u=a$ and $v=u +\epsilon$, thus
	$$\left\vert {\mathcal{T}_a^{v} \left( {f,g} \right) -
		\mathcal{T}_{u}^b \left( {f,g} \right)}\right\vert \longrightarrow
	\left\vert {\mathcal{T}_a^{b} \left( {f,g} \right)}\right\vert
	\,\,\,\,\,\, \text{as} \,\,\,\,\,\,\epsilon \longrightarrow 0^+
	.$$ Consequently, the right hand of \eqref{eq2.2}
	$\longrightarrow$ the right hand of \eqref{eq2.1}.
	
\end{remark}

\section{Bounds for bounded variation integrators}

The first result regarding bounded variation integrators is
presented as follows:
\begin{thm}
	\label{thm4.5.1}Let $f,g:[a,b]\rightarrow \mathbb{R}$ be such that
	$f$
	is of bounded variation on $[a,b]$ and $g$ is absolutely continuous on $%
	[a,b] $, then
	\begin{multline}
	\left\vert {\mathcal{T}_a^v \left( {f,g} \right) -
		\mathcal{T}_u^b \left(
		{f,g} \right)}\right\vert  \label{eq4.5.2} \\
	\leq \bigvee_{a}^{b}\left( {f}\right) \cdot \left\{
	\begin{array}{l}
	\frac{1}{8} \left[ {\frac{{\left( {v - a} \right) + \left( {b - u}
				\right)}}{2} + \left| {\frac{{b - u}}{2} - \frac{{v - a}}{2}}
		\right|} \right]  \left\Vert {g^{\prime }}\right\Vert
	_{\infty,[a,b] }, \,\,\,\,\,{\rm{if}}\,\,\,\,\,\,g^{\prime }\in
	L_{\infty }\left[ {a,b}\right] ; \\
	\\
	\frac{1}{{2\left( {q+1}\right) ^{1/q}
	}}\left[ {\frac{{b-a}}{2} + \left| {v - \frac{{a +
					b}}{2}} \right|} \right]\cdot \left\Vert {g^{\prime }}\right\Vert _{p,[a,b]}, \,\,\,\,\,\,\,\,\,\,\,\,{\rm{if}}\,\,\,\,\,\,g^{\prime }\in L_{p}\left[ {%
		a,b}\right], \\
	\\
	\\
	\frac{1}{2}\left\Vert {g^{\prime }}\right\Vert
	_{1,[a,b]},\,\,\,\,\,\,\,\,\,\,\,\,\,\,\,\,\,\,\,\,\,\,\,\,\,\,\,\,\,\,\,\,\,\,\,\,\,\,\,\,\,\,\,\,\,\,\,\,\,\,\,\,\,\,\,\,\,\,\,\,\,\,\,\,\,\,\,\,\,\,\,\,\,\,\,\,{\rm{if}}\,\,\,\,\,g^{\prime
	}\in L_{1}\left[ {a,b}\right],%
	\end{array}%
	\right.,
	\end{multline}
	for all $a\le u<v\le b$, where $\left\Vert {\cdot }\right\Vert
	_{p}$ are the usual Lebesgue norms,
	i.e.,%
	\begin{align*}
	\left\Vert {h}\right\Vert _{p}:=\left( \int_{a}^{b}\left\vert
	h\left( t\right) \right\vert ^{p}dt\right) ^{1/p},\text{ for
	}p\geq 1
	\end{align*}%
	and%
	\begin{align*}
	\left\Vert {h}\right\Vert _{\infty }:=ess\mathop {\sup
	}\limits_{t\in \left[ {a,b}\right] }\left\vert {h\left( t\right)
	}\right\vert .
	\end{align*}
\end{thm}

\begin{proof}
	It is known that for a continuous function $w$ on $[a,b]$ and a
	bounded variation $\nu$ on $[a,b]$, one have the inequality
	\begin{align}
	\label{eq4.3.17}\left| {\int_a^b {w\left( t \right)d\nu\left( t
			\right)} } \right| \le  \mathop {\sup }\limits_{t \in \left[ {a,b}
		\right]} \left|{w\left( t \right)}\right|
	\bigvee_a^b\left({\nu}\right).
	\end{align}
	Employing \eqref{eq4.3.17} for the Cerone-Dragomir identity
	\begin{align}
	\mathcal{T}\left( {f,g} \right) =  - \frac{1}{{b - a}}\int_a^b
	{\left( {\int_a^t {g\left( s \right)ds}  - \frac{{t - a}}{{b -
					a}}\int_a^b {g\left( s \right)ds} } \right)df\left( t \right)}.
	\label{eq4.5.3}
	\end{align}%
	One has as $f$ is of bounded variation on $[a,b]$,
	\begin{align}
	&\left\vert {\mathcal{T}_a^v \left( {f,g} \right) -
		\mathcal{T}_u^b \left( {f,g} \right)}\right\vert
	\nonumber\\
	& = \left| { \frac{1}{{v - a}}\int_a^v {\left( {\int_a^t {g\left(
					s \right)ds}  - \frac{{t - a}}{{v - a}}\int_a^v {g\left( s
					\right)ds} } \right)df\left( t \right)} } \right.
	\nonumber\\
	&\qquad\left. {- \frac{1}{{b - u}}\int_u^b {\left( {\int_u^r
				{g\left( s \right)ds}  - \frac{{r - u}}{{b - u}}\int_u^b {g\left(
					s \right)ds} } \right)df\left( t \right)} } \right|\nonumber
	\nonumber\\
	&\le \left| {\frac{1}{{v - a}}\int_a^v {\left( {\int_a^t {g\left(
					s \right)ds}  - \frac{{t - a}}{{v - a}}\int_a^v {g\left( s
					\right)ds} } \right)df\left( t \right)}  } \right|\nonumber
	\\
	&\qquad+ \left| {\frac{1}{{b - u}}\int_u^b {\left( {\int_u^r
				{g\left( s \right)ds}  - \frac{{r - u}}{{b - u}}\int_u^b {g\left(
					s \right)ds} } \right)df\left( t \right)} } \right|\nonumber
	\\
	&\le\frac{1}{{v - a}}\mathop {\sup }\limits_{t \in \left[ {a,v}
		\right]} \left| { \int_a^v {\left( {\int_a^t {g\left( s \right)ds}
				- \frac{{t - a}}{{v - a}}\int_a^v {g\left( s \right)ds} }
			\right)dt} } \right| \cdot \bigvee_a^v \left( f \right) \nonumber
	\\
	&\qquad+ \frac{1}{{b - u}}\mathop {\sup }\limits_{r \in \left[
		{u,b} \right]} \left| { \int_u^b {\left( {\int_u^r {g\left( s
					\right)ds}  - \frac{{r - u}}{{b - u}}\int_u^b {g\left( s
					\right)ds} } \right)dt}  } \right| \cdot \bigvee_u^b \left( f
	\right)\label{eq4.5.4}
	\end{align}%
	In the inequality (\ref{Bar4.3.1}), setting $d=t$, $c=a$ and then
	$d=r$, $c=u$, we get
	\begin{align}
	\left\vert {\frac{1}{{t-a}}\int_{a}^{t}{g\left( s\right) ds}-\frac{1}{{v-a}}%
		\int_{a}^{v}{g\left( s\right) ds}}\right\vert \leq \frac{1}{2}\left( {v-t}%
	\right) \left\Vert {g^{\prime }}\right\Vert _{\infty,[a,v] }
	\label{eq4.5.5}
	\end{align}
	and
	\begin{align}
	\label{eq4.5.6} \left\vert {\frac{1}{{r-u}}\int_{u}^{r}{g\left( s\right) ds}-\frac{1}{{b-u}}%
		\int_{u}^{b}{g\left( s\right) ds}}\right\vert \leq \frac{1}{2}\left( {b-r}%
	\right) \left\Vert {g^{\prime }}\right\Vert _{\infty,[u,b] }.
	\end{align}
	Substituting (\ref{eq4.5.5}) and (\ref{eq4.5.6}) in
	(\ref{eq4.5.4}), we get
	\begin{align}
	\left\vert {\mathcal{T}_a^v \left( {f,g} \right) - \mathcal{T}_u^b
		\left( {f,g} \right)}\right\vert & \leq \frac{1}{{v - a}}\cdot\frac{1}{2}\left\Vert {g^{\prime }}%
	\right\Vert _{\infty,[a,v] }\mathop {\sup }\limits_{t\in \left[
		{a,v}\right] }\left\{ {\left( {t-a}\right) \left( {v-t}\right)
	}\right\} \bigvee_{a}^{v}\left( {f}\right)
	\nonumber\\
	&\qquad+
	\frac{1}{{b - u}}\cdot\frac{1}{2}\left\Vert {g^{\prime }}%
	\right\Vert _{\infty,[u,b] }\mathop {\sup }\limits_{r\in \left[
		{u,b}\right] }\left\{ {\left( {r-u}\right) \left( {b-r}\right)
	}\right\}\bigvee_{u}^{b}\left( {f}\right)
	\nonumber\\
	& =\frac{1}{8}\left( {v-a}\right)\left\Vert {g^{\prime
	}}\right\Vert _{\infty,[a,v] }\bigvee_{a}^{v}\left( {f}\right)+
	\frac{1}{8}\left( {b-u}\right)\left\Vert {g^{\prime }}\right\Vert
	_{\infty,[u,b] }\bigvee_{u}^{b}\left( {f}\right) \label{eq4.5.7}
	\\
	&\le \frac{1}{8} \max \left\{ {\left( {v - a} \right),\left( {b -
			u} \right)} \right\} \left\Vert {g^{\prime }}\right\Vert
	_{\infty,[a,b] }\bigvee_{a}^{b}\left( {f}\right)
	\nonumber\\
	&\le \frac{1}{8} \left[ {\frac{{\left( {v - a} \right) + \left( {b
					- u} \right)}}{2} + \left| {\frac{{b - u}}{2} - \frac{{v - a}}{2}}
		\right|} \right]  \left\Vert {g^{\prime }}\right\Vert
	_{\infty,[a,b] }\bigvee_{a}^{b}\left( {f}\right)\nonumber
	\end{align}%
	where we used the fact that $\mathop {\sup }\limits_{t\in \left[
		{\alpha,\beta}\right] }\left\{ {\left( {t-\alpha }\right) \left(
		{\beta-t}\right) }\right\} $, occurs at
	$t=\frac{\alpha+\beta}{2}$,
	therefore, $\mathop {\sup }\limits_{t\in \left[ {\alpha,\beta}\right] }\left\{ {%
		\left( {t-\alpha}\right) \left( {\beta-t}\right) }\right\} =\frac{1}{4}\left( {\beta-\alpha}%
	\right) ^{2}$. Also, we note that the last inequality holds since
	\begin{align*}
	\left\Vert
	{g^{\prime }}\right\Vert _{\infty,[a,v] }\le \left\Vert {g^{\prime
	}}\right\Vert _{\infty,[a,b] },\,\,\,\bigvee_{a}^{v}\left(
	{f}\right) \le \bigvee_{a}^{b}\left(
	{f}\right)\,\,\,\text{and}\,\,\,\bigvee_{u}^{b}\left( {f}\right)
	\le \bigvee_{a}^{b}\left( {f}\right),
	\end{align*}
	which proves the first inequality in (\ref{eq4.5.2}).
	
	In the inequality (\ref{Cer4.3.2}), replace $r,u$ instead of
	$d,c$; respectively and then $t,a$ instead of $d,c$; respectively,
	we find that
	\begin{align}
	& \left\vert {\frac{1}{{r-a}}\int_{a}^{r}{g\left( s\right) ds}-\frac{1}{{v-a}%
		}\int_{a}^{v}{g\left( s\right) ds}}\right\vert  \label{eq4.5.8} \\
	& \leq \left\{
	\begin{array}{l}
	\frac{\left( {v-r}\right)^{\frac{1}{q}}}{{\left( {q+1}%
			\right) ^{1/q}\left( {v-a}\right) ^{\frac{1}{q}}}}\left[ {\left(
		{r - a} \right)^q  + \left( {v - r} \right)^q } \right]^{1/q}
	\left\Vert {%
		g^{\prime }}\right\Vert _{p,[a,v]},\,\,\,\,\,\,g^{\prime }\in L_{p}\left[ {a,v}%
	\right] , \\
	\\
	\frac{{v-r}}{{v-a}}\left\Vert {g^{\prime }}\right\Vert
	_{1,[a,v]},\,\,\,\,\,\,\,\,\,\,\,\,\,\,\,\,g^{\prime }\in
	L_{1}\left[ {a,v}\right].
	\end{array}%
	\right.  \nonumber
	\end{align}%
	and
	\begin{align}
	& \left\vert {\frac{1}{{t-u}}\int_{u}^{t}{g\left( s\right) ds}-\frac{1}{{b-a}%
		}\int_{u}^{b}{g\left( s\right) ds}}\right\vert  \label{eq4.5.9} \\
	& \leq \left\{
	\begin{array}{l}
	\frac{\left( {b-t}\right)^{\frac{1}{q}}}{{\left( {q+1}%
			\right) ^{1/q}\left( {b-u}\right) ^{\frac{1}{q}}}}\left[ {\left(
		{t - u} \right)^q  + \left( {b - t} \right)^q } \right]^{1/q}
	\left\Vert {%
		g^{\prime }}\right\Vert _{p,[u,b]},\,\,\,\,\,\,g^{\prime }\in L_{p}\left[ {u,b}%
	\right] , \\
	\\
	\frac{{b-t}}{{b-u}}\left\Vert {g^{\prime }}\right\Vert
	_{1,[u,b]},\,\,\,\,\,\,\,\,\,\,\,\,\,\,\,\,g^{\prime }\in
	L_{1}\left[ {u,b}\right]
	\end{array}%
	\right.  \nonumber
	\end{align}%
	Substituting (\ref{eq4.5.8}) and (\ref{eq4.5.9}) in
	(\ref{eq4.5.4}), we have respectively
	\begin{align}
	&\frac{1}{v-a}  \mathop {\sup }\limits_{r \in \left[ {a,v}
		\right]} \left( {r-a}\right) \left| {\frac{1}{{r-a}}%
		\int_{a}^{r}{g\left( s\right)
			ds}-\frac{1}{{v-a}}\int_{a}^{v}{g\left( s\right) ds}}\right|\nonumber \\
	& \leq \frac{1}{(v-a)}\left\{
	\begin{array}{l}
	\frac{{\left\Vert {g^{\prime }}\right\Vert _{p,[a,v]}}}{{\left(
			{q+1}\right) ^{1/q}\left( {v-a}\right) ^{1/q}}} \mathop {\sup
	}\limits_{t \in \left[ {a,v}
		\right]}\left\{\left( {r-a}\right) \left( {%
		v-r}\right) ^{{\textstyle{\frac{1}{q}}}}\left[ {\left( {r - a}
		\right)^q  + \left( {v - r}
		\right)^q } \right]^{1/q}\right\},\,\,g^{\prime }\in L_{p}\left[ {a,v}%
	\right] , \\ \\
	\\
	\frac{\left\Vert {g^{\prime }}\right\Vert _{1,[a,v]}}{{v-a}}
	\mathop {\sup }\limits_{r \in \left[ {a,v}
		\right]}\left( {r-a}\right) \left( {v-r}\right),\,\,\,\,\,\,\,\,\,\,\,\,\,\,\,%
	\,g^{\prime }\in L_{1}\left[ {a,v}\right] ,%
	\end{array}%
	\right. \nonumber\\
	& =\left\{
	\begin{array}{l}
	\frac{\left( {v-a}\right)}{{4\left( {q+1}\right) ^{1/q}
	}}\cdot \left\Vert {g^{\prime }}\right\Vert _{p,[a,v]},\,\,\,\,\,\,\,\,\,\,\,\,\,\,\,g^{\prime }\in L_{p}\left[ {%
		a,v}\right], \\
	\\
	\\
	\frac{1}{4}\left\Vert {g^{\prime }}\right\Vert
	_{1,[a,v]},\,\,\,\,\,\,\,\,\,\,\,\,\,\,\,\,\,\,\,\,\,\,\,\,\,\,\,\,\,\,\,\,\,\,\,\,g^{\prime
	}\in L_{1}\left[
	{a,v}\right] , \\
	\end{array}%
	\right.,\label{eq4.5.10}
	\end{align}%
	and similarly, we have
	\begin{align}
	&\frac{1}{b-u}  \mathop {\sup }\limits_{r \in \left[ {u,b}
		\right]} \left( {r-u}\right) \left| {\frac{1}{{t-u}}%
		\int_{u}^{t}{g\left( s\right)
			ds}-\frac{1}{{b-u}}\int_{u}^{b}{g\left( s\right) ds}}\right|\label{eq4.4.11}\\
	& \leq  \left\{
	\begin{array}{l}
	\frac{\left( {b-u}\right)}{{4\left( {q+1}\right) ^{1/q}
	}}\cdot \left\Vert {g^{\prime }}\right\Vert _{p,[u,b]},\,\,\,\,\,\,\,\,\,\,\,\,\,\,\,g^{\prime }\in L_{p}\left[ {%
		u,b}\right], \\
	\\
	\\
	\frac{1}{4}\left\Vert {g^{\prime }}\right\Vert
	_{1,[u,b]},\,\,\,\,\,\,\,\,\,\,\,\,\,\,\,\,\,\,\,\,\,\,\,\,\,\,\,\,\,\,\,\,\,\,\,\,g^{\prime
	}\in L_{1}\left[
	{u,b}\right]. \\
	\end{array}%
	\right.\nonumber
	\end{align}%
	Adding (\ref{eq4.5.10}) and (\ref{eq4.4.11}),  we get
	\begin{align*}
	&\left\vert {\mathcal{T}_a^v \left( {f,g} \right) -
		\mathcal{T}_u^b \left( {f,g} \right)}\right\vert
	\\
	&\le\left\{
	\begin{array}{l}
	\frac{\left( {v-a}\right)}{{4\left( {q+1}\right) ^{1/q}
	}}\cdot \left\Vert {g^{\prime }}\right\Vert _{p,[a,v]}\bigvee_{a}^{v}\left( {f}\right)+\frac{\left( {b-u}\right)}{{4\left( {q+1}\right) ^{1/q}
	}}\cdot \left\Vert {g^{\prime }}\right\Vert _{p,[u,b]}\bigvee_{u}^{b}\left( {f}\right), \,\,\,\,\,\,\,\,\,g^{\prime }\in L_{p}\left[ {%
		u,b}\right], \\
	\\
	\\
	\frac{1}{4}\left\Vert {g^{\prime }}\right\Vert
	_{1,[a,v]}\bigvee_{a}^{v}\left( {f}\right)+\frac{1}{4}\left\Vert
	{g^{\prime }}\right\Vert _{1,[u,b]}\bigvee_{u}^{b}\left(
	{f}\right),\,\,\,\,\,\,\,\,\,\,\,\,\,\,\,\,\,\,\,\,\,\,\,\,\,\,\,\,\,\,\,\,\,\,\,\,g^{\prime
	}\in L_{1}\left[ {u,b}\right] ,
	\end{array}%
	\right.
	\\
	\\
	&\le\left\{
	\begin{array}{l}
	\frac{1}{{2\left( {q+1}\right) ^{1/q}
	}}\left[ {\frac{{b-a}}{2} + \left| {v - \frac{{a +
					b}}{2}} \right|} \right]\cdot \left\Vert {g^{\prime }}\right\Vert _{p,[a,b]}\bigvee_{a}^{b}\left( {f}\right), \,\,\,\,\,\,\,\,\,g^{\prime }\in L_{p}\left[ {%
		a,b}\right], \\
	\\
	\\
	\frac{1}{2}\left\Vert {g^{\prime }}\right\Vert
	_{1,[a,b]}\bigvee_{a}^{b}\left(
	{f}\right),\,\,\,\,\,\,\,\,\,\,\,\,\,\,\,\,\,\,\,\,\,\,\,\,\,\,\,\,\,\,\,\,\,\,\,\,g^{\prime
	}\in L_{1}\left[ {a,b}\right],
	\end{array}%
	\right.
	\end{align*}
	which proves the second and the third inequalities in
	(\ref{eq4.5.2})
\end{proof}

\begin{cor}
	Under the assumptions of Theorem \ref{thm4.5.1}, we have
	\begin{align}
	& \left\vert {\mathcal{T}_a^u \left( {f,g} \right) -
		\mathcal{T}_u^b \left(
		{f,g} \right)}\right\vert  \label{eq4.5.12} \\
	& \leq \bigvee_{a}^{b}\left( {f}\right) \cdot \left\{
	\begin{array}{l}
	\frac{1}{8} \left[ {\frac{{b-a}}{2} + \left| {u - \frac{{a +
					b}}{2}} \right|} \right]  \left\Vert {g^{\prime }}\right\Vert
	_{\infty,[a,b] },
	\,\,\,\,\,\,\,\,\,\,\,\,\,\,\,\,\,\,\,{\rm{if}}\,\,\,\,\,\,g^{\prime
	}\in
	L_{\infty }\left[ {a,b}\right] ; \\
	\\
	\frac{1}{{2\left( {q+1}\right) ^{1/q}
	}}\left[ {\frac{{b-a}}{2} + \left| {u - \frac{{a +
					b}}{2}} \right|} \right] \left\Vert {g^{\prime }}\right\Vert _{p,[a,b]},\,\,\,\,\,{\rm{if}}\,\,\,\,\,\,g^{\prime }\in L_{p}\left[ {%
		a,b}\right], \\
	\\
	\frac{1}{2}\left\Vert {g^{\prime }}\right\Vert
	_{1,[a,b]},\,\,\,\,\,\,\,\,\,\,\,\,\,\,\,\,\,\,\,\,\,\,\,\,\,\,\,\,\,\,\,\,\,\,\,\,\,\,\,\,\,\,\,\,\,\,\,\,\,\,\,\,\,\,\,\,\,\,\,\,\,\,\,\,\,\,\,{\rm{if}}\,\,\,\,\,g^{\prime
	}\in L_{1}\left[ {a,b}\right].
	\end{array}%
	\right.  \nonumber
	\end{align}
	for all $a\le u\le b$. In particular case if $u=\frac{a+b}{2}$, we
	get
	\begin{align}
	& \left\vert {\mathcal{T}_a^{\frac{a+b}{2}} \left( {f,g} \right) -
		\mathcal{T}_{\frac{a+b}{2}}^b \left(
		{f,g} \right)}\right\vert  \label{eq4.5.13} \\
	& \leq \bigvee_{a}^{b}\left( {f}\right) \cdot \left\{
	\begin{array}{l}
	\frac{b-a}{16}    \left\Vert {g^{\prime }}\right\Vert
	_{\infty,[a,b] },
	\,\,\,\,\,\,\,\,\,\,\,\,\,\,\,\,\,\,\,\,\,{\rm{if}}\,\,\,\,\,g^{\prime
	}\in
	L_{\infty }\left[ {a,b}\right] ; \\
	\\
	\frac{b-a}{{4\left( {q+1}\right) ^{1/q}
	}} \cdot \left\Vert {g^{\prime }}\right\Vert _{p,[a,b]},    \,\,\,\,\,\,\,\,\,{\rm{if}}\,\,\,\,\,\,g^{\prime }\in L_{p}\left[ {%
		a,b}\right], \\
	\\
	\\
	\frac{1}{2}\left\Vert {g^{\prime }}\right\Vert _{1,[a,b]},
	\,\,\,\,\,\,\,\,\,\,\,\,\,\,\,\,\,\,\,\,\,\,\,\,\,\,\,\,\,\,{\rm{if}}\,\,\,\,\,g^{\prime
	}\in L_{1}\left[ {a,b}\right].
	\end{array}%
	\right.  \nonumber
	\end{align}
\end{cor}

\begin{proof}
	In Theorem \ref{thm4.5.1}, let $\epsilon>0$ and set $v=u+\epsilon$
	so as $\epsilon \to 0^+$ we get the required result.
\end{proof}

Another result when $g$ is of $r$-$H$--H\"{o}lder type is as
follows:
\begin{thm}
	\label{thm4.5.3} Let $f,g:[a,b]\rightarrow \mathbb{R}$ be such
	that $f$ is of bounded variation on $[a,b]$ and $g$ is of
	$p$-$H$--H\"{o}lder type on $[a,b]$, for $p\in(0,1]$ and $H>0$ are
	given. Then
	\begin{align}
	\left\vert {\mathcal{T}_a^{v} \left( {f,g} \right) -
		\mathcal{T}_{u}^b \left( {f,g} \right)}\right\vert
	\label{eq4.5.14}   \le  H\frac{{\left( {v - a} \right)^p  + \left(
			{b - u} \right)^p }}{{2^{p + 1} \left( {p + 1} \right)}}
	\bigvee_{a}^{b}\left( {%
		f}\right),
	\end{align}
	and
	\begin{align}
	\left\vert {\mathcal{T}_a^{v} \left( {f,g} \right) -
		\mathcal{T}_{u}^b \left( {f,g} \right)}\right\vert
	\label{eq4.5.15}
	\le \frac{{H}}{{2^{p} \left( {p + 1} \right)}} \left[
	{\frac{{\left( {v - a} \right) + \left( {b - u} \right)}}{2} +
		\left| {\frac{{v - a}}{2} - \frac{{b - u}}{2}} \right|} \right]^p
	\cdot\bigvee_{a}^{b}\left( { f}\right),
	\end{align}
	for all $a\le u<v\le b$.
\end{thm}

\begin{proof}
	We repeat the proof of Theorem \ref{thm4.5.1}. So as $f$ is of
	bounded variation and $g$ is of $p$-$H$--H\"{o}lder type on
	$[a,b]$, then we have
	\begin{align*}
	&\left\vert {\mathcal{T}_a^{v} \left( {f,g} \right) -
		\mathcal{T}_{u}^b \left( {f,g} \right)}\right\vert
	\nonumber\\
	& \leq
	\frac{1}{{v - a}}\mathop {\sup }\limits_{r\in \left[ {a,v}\right] }\left\vert {\left( {%
			r-a}\right) \left[ {\frac{1}{{r-a}}\int_{a}^{r}{g\left( s\right) ds}-\frac{1%
			}{{v-a}}\int_{a}^{v}{g\left( s\right) ds}}\right] }\right\vert
	\bigvee_{a}^{v}\left( {f}\right)
	\nonumber\\
	&\qquad+ \frac{1}{{b - u}}\mathop {\sup }\limits_{t\in \left[ {u,b}\right] }\left\vert {\left( {%
			t-u}\right) \left[ {\frac{1}{{t-u}}\int_{u}^{t}{g\left( s\right) ds}-\frac{1%
			}{{b-u}}\int_{u}^{b}{g\left( s\right) ds}}\right] }\right\vert
	\bigvee_{u}^{b}\left( {f}\right)
	\nonumber\\
	& \leq \frac{1}{{v - a}}\frac{H}{p+1}\mathop {\sup }\limits_{t\in \left[ {a,v}\right] }{%
		\left( {r-a}\right) \left( {v-r}\right) ^{p}}\bigvee_{a}^{v}\left( {%
		f}\right)+\frac{1}{{b - u}}\frac{H}{p+1}\mathop {\sup }\limits_{t\in \left[ {u,b}\right] }{%
		\left( {t-u}\right) \left( {b-t}\right) ^{p}}\bigvee_{u}^{b}\left( {%
		f}\right)
	\nonumber\\
	& =  H\frac{{\left( {v - a} \right)^p }}{{2^{p + 1} \left( {p + 1}
			\right)}} \bigvee_{a}^{v}\left( {%
		f}\right)+H\frac{{\left( {b - u} \right)^p }}{{2^{p + 1} \left( {p
				+ 1}\right)}} \bigvee_{u}^{b}\left( {%
		f}\right)
	\\
	& \le   H\frac{{\left( {v - a} \right)^p  + \left( {b - u}
			\right)^p }}{{2^{p + 1} \left( {p + 1} \right)}}
	\bigvee_{a}^{b}\left( {%
		f}\right),\nonumber
	\end{align*}
	which proves the first inequality. To obtain the second inequality
	from the above inequality we may obtain that
	\begin{align*}
	&\left\vert {\mathcal{T}_a^{v} \left( {f,g} \right) -
		\mathcal{T}_{u}^b \left( {f,g} \right)}\right\vert
	\\
	&\le H\frac{{\left( {v - a} \right)^p }}{{2^{p + 1} \left( {p + 1}
			\right)}} \bigvee_{a}^{v}\left( {%
		f}\right)+H\frac{{\left( {b - u} \right)^p }}{{2^{p + 1} \left( {p
				+ 1}
			\right)}} \bigvee_{u}^{b}\left( {%
		f}\right)
	\\
	&\le H\frac{{1}}{{2^{p + 1} \left( {p + 1}
			\right)}} \max\left\{\left( {v - a} \right)^p ,\left( {b - u} \right)^p \right\}  \left[\bigvee_{a}^{v}\left( {%
		f}\right)+\bigvee_{u}^{b}\left( { f}\right) \right]
	\\
	&\le H\frac{{1}}{{2^{p} \left( {p + 1} \right)}} \left[
	{\frac{{\left( {v - a} \right) + \left( {b - u} \right)}}{2} +
		\left| {\frac{{v - a}}{2} - \frac{{b - u}}{2}} \right|} \right]^p
	\cdot\bigvee_{a}^{b}\left( { f}\right).
	\end{align*}
	which proves (\ref{eq4.5.15}), and thus the proof is completed.
\end{proof}

\begin{cor}
	Under the assumptions of Theorem \ref{thm4.5.3}, we have
	\begin{align}
	\left\vert {\mathcal{T}_a^{u} \left( {f,g} \right) -
		\mathcal{T}_{u}^b \left( {f,g} \right)}\right\vert
	\label{eq4.5.17}   \le  H\frac{{\left( {u - a} \right)^p  + \left(
			{b - u} \right)^p }}{{2^{p + 1} \left( {p + 1} \right)}}
	\bigvee_{a}^{b}\left( {%
		f}\right),
	\end{align}
	and
	\begin{align}
	\left\vert {\mathcal{T}_a^{u} \left( {f,g} \right) -
		\mathcal{T}_{u}^b \left( {f,g} \right)}\right\vert
	\label{eq4.5.18}    \le \frac{{H}}{{2^{p} \left( {p + 1} \right)}}
	\left[ {\frac{{b-a}}{2} + \left| {u-\frac{{a + b}}{2}} \right|}
	\right]^p \cdot\bigvee_{a}^{b}\left( { f}\right),
	\end{align}
	for all $a\le u \le b$. In particular case if $u=\frac{a+b}{2}$,
	then the both inequalities (\ref{eq4.5.17}) and (\ref{eq4.5.18})
	gives the same inequality, that is
	\begin{align}
	& \left\vert {\mathcal{T}_a^{\frac{a+b}{2}} \left( {f,g} \right) -
		\mathcal{T}_{\frac{a+b}{2}}^b \left( {f,g} \right)}\right\vert
	\label{eq4.5.19}\le  H\frac{{\left( {b - a} \right)^p  }}{{2^{2p}
			\left( {p + 1} \right)}} \bigvee_{a}^{b}\left( { f}\right).
	\end{align}
\end{cor}

\begin{proof}
	In Theorem \ref{thm4.5.3}, let $\epsilon>0$ and set $v=u+\epsilon$
	so as $\epsilon \to 0^+$ we get the required result.
\end{proof}

\begin{thm}
	\label{thm4.5.5} Let $f,g:[a,b]\rightarrow \mathbb{R}$ be such
	that
	$f$ is of bounded variation on $[a,b]$ and $g$ is monotonic nondecreasing on $%
	[a,b]$, then
	\begin{multline}
	\left\vert {\mathcal{T}_a^{v} \left( {f,g} \right) -
		\mathcal{T}_{u}^b \left( {f,g} \right)}\right\vert
	\label{eq4.5.20}
	\\
	\le\frac{1}{4} \left\{ {\frac{{\left[ {g\left( v \right) - g\left(
					a \right)} \right] + \left[ {g\left( b \right) - g\left( u
					\right)} \right]}}{2} + \left| {  \frac{{g\left( v \right) +
					g\left( u \right)}}{2} - \frac{{g\left( a \right) + g\left( b
					\right)}}{2} } \right|} \right\}\cdot \bigvee_a^b \left( f
	\right),
	\end{multline}
	for all $a\le u<v\le b$.
\end{thm}

\begin{proof}
	As $f$  is of bounded variation on $[a,b]$ and $g$ is monotonic
	nondecreasing on $[a,b]$ (which implies that $\Psi_g\left( {t;a,b}
	\right)$ is absolutely continuous on $[a,b]$), by (\ref{eq4.5.4})
	we have
	\begin{align}
	& \left\vert {\mathcal{T}_a^{v} \left( {f,g} \right) -
		\mathcal{T}_{u}^b \left( {f,g} \right)}\right\vert  \nonumber \\
	& \leq \frac{1}{{v - a}} \mathop {\sup }\limits_{r \in \left[
		{a,v} \right]}
	\left[ {\left( {r-a}\right) \left| {\frac{1}{{r-a}}%
			\int_{a}^{r}{g\left( s\right)
				ds}-\frac{1}{{v-a}}\int_{a}^{b}{g\left( s\right) ds}}\right|
	}\right]  \bigvee_{a}^{v}\left( f\right) \nonumber
	\\
	&\qquad+ \frac{1}{{b - u}} \mathop {\sup }\limits_{r \in \left[
		{u,b} \right]} \left[ {\left( {t-u}\right) \left| {\frac{1}{{t-u}}%
			\int_{u}^{t}{g\left( s\right)
				ds}-\frac{1}{{b-u}}\int_{u}^{b}{g\left( s\right) ds}}\right|
	}\right]  \bigvee_{u}^{b}\left( f\right).\label{eq4.5.21}
	\end{align}%
	Employing the third part of (\ref{Cer4.3.4}), setting $d=r,t$ and
	$c=a,u$, respectively  we get
	\begin{align}
	\left\vert {\frac{1}{{r-a}}\int_{a}^{r}{g\left( s\right)
			ds}-\frac{1}{{v-a}} \int_{a}^{v}{g\left( s\right) ds}}\right\vert
	\leq \frac{{v-r}}{{v-a}}\left[ {g\left( v\right) -g\left( a\right)
	}\right] .  \label{eq4.5.22}
	\end{align}
	and
	\begin{align}
	\left\vert {\frac{1}{{t-u}}\int_{u}^{t}{g\left( s\right) ds}-\frac{1}{{b-u}}%
		\int_{u}^{b}{g\left( s\right) ds}}\right\vert \leq
	\frac{{b-t}}{{b-u}}\left[ {g\left( b\right) -g\left( u\right)
	}\right] .  \label{eq4.5.23}
	\end{align}%
	Substituting (\ref{eq4.5.22}) and (\ref{eq4.5.23}) in
	(\ref{eq4.5.21}), we get
	\begin{align*}
	&\left\vert {\mathcal{T}_a^{v} \left( {f,g} \right) -
		\mathcal{T}_{u}^b \left( {f,g} \right)}\right\vert
	\\
	& \leq \frac{1}{{\left( {v - a} \right)^2 }}\mathop {\sup
	}\limits_{r \in \left[ {a,v} \right]} \left\{ {\left( {r - a}
		\right)\left( {v - r} \right)} \right\} \cdot \left[ {g\left( v
		\right) - g\left( a \right)} \right]\bigvee_a^v \left( f \right)
	\\
	&\qquad+   \frac{1}{{\left( {b - u} \right)^2 }}\mathop {\sup
	}\limits_{t \in \left[ {u,b} \right]} \left\{ {\left( {t - u}
		\right)\left( {b - t} \right)} \right\} \cdot \left[ {g\left( b
		\right) - g\left( u \right)} \right]\bigvee_u^b \left( f \right)
	\\
	&=\frac{1}{4}\left[ {g\left( v \right) - g\left( a \right)}
	\right]\bigvee_a^v \left( f \right)+ \frac{1}{4}\left[ {g\left( b
		\right) - g\left( u \right)} \right]\bigvee_u^b \left( f \right)
	\\
	&= \frac{1}{4}\max\{{g\left( v \right) - g\left( a \right) ,
		g\left( b \right) - g\left( u \right)}\}\cdot \bigvee_a^b \left( f
	\right)
	\\
	&\le \frac{1}{4} \left\{ {\frac{{\left[ {g\left( v \right) -
					g\left( a \right)} \right] + \left[ {g\left( b \right) - g\left( u
					\right)} \right]}}{2} + \left| {  \frac{{g\left( v \right) +
					g\left( u \right)}}{2} - \frac{{g\left( a \right) + g\left( b
					\right)}}{2} } \right|} \right\}\cdot \bigvee_a^b \left( f
	\right),
	\end{align*}
	and thus the proof is finished.
\end{proof}
\begin{cor}
	Under the assumptions of Theorem \ref{thm4.5.5}, we have
	\begin{align}
	\left\vert {\mathcal{T}_a^{u} \left( {f,g} \right) -
		\mathcal{T}_{u}^b \left( {f,g} \right)}\right\vert
	\label{eq4.5.24} \le  \left\{ {\frac{{g\left( b \right) - g\left(
				a \right)}}{2} + \left| {g\left( u \right) - \frac{{g\left( a
					\right) + g\left( b \right)}}{2}} \right|} \right\}\cdot
	\bigvee_a^b \left( f \right),
	\end{align}
	for all $a\le u\le b$. In particular case if $u=\frac{a+b}{2}$,
	then the both inequalities (\ref{eq4.5.24}) gives the same
	inequality, that is
	\begin{multline}
	\left\vert {\mathcal{T}_a^{\frac{a+b}{2}} \left( {f,g} \right) -
		\mathcal{T}_{\frac{a+b}{2}}^b \left( {f,g} \right)}\right\vert
	\label{eq4.5.25}
	\\
	\le  \left\{ {\frac{{g\left( b \right) - g\left( a \right)}}{2} +
		\left| {g\left( \frac{a+b}{2} \right) - \frac{{g\left( a \right) +
					g\left( b \right)}}{2}} \right|} \right\}\cdot \bigvee_a^b \left(
	f \right),
	\end{multline}
\end{cor}
\begin{proof}
	In Theorem \ref{thm4.5.3}, let $\epsilon>0$ and set $v=u+\epsilon$
	so as $\epsilon \to 0^+$ we get the required result.
\end{proof}

\section{Bounds for Lipschitzian integrators}

\begin{thm}
	\label{thm4.5.7} Let $f,g:[a,b]\rightarrow \mathbb{R}$ be such
	that $f$
	is $L$--Lipschitzian on $[a,b]$ and $g$ is an absolutely continuous on $%
	[a,b] $, then
	\begin{align}
	&\left\vert {\mathcal{T}_a^{v} \left( {f,g} \right) -
		\mathcal{T}_{u}^b \left( {f,g} \right)}\right\vert  \label{eq4.4.26} \\
	& \leq L\left\{
	\begin{array}{l}
	\frac{{\left[ {\left( {b - a} \right) - \left( {v - u} \right)}
			\right]}}{6} \cdot \left[ {\frac{1}{4} + \left( {\frac{{\frac{{a +
							b}}{2} - \frac{{u + v}}{2}}}{{\left( {b - a} \right) - \left( {v -
						u} \right)}}} \right)^2 } \right]
	\left\Vert {g^{\prime }}\right\Vert _{\infty } ,\,\,\,\,\,\,\,\,\,\,\,\,\,\,\,\,g^{\prime }\in L_{\infty }\left[ {%
		a,b}\right] ; \\
	\\
	\frac{{2\left[ {\left( {b - a} \right) - \left( {v - u} \right)}
			\right]}}{{\left( {q+1}\right) ^{1/q}}}   \cdot \left[
	{\frac{1}{4} + \left( {\frac{{\frac{{a + b}}{2} - \frac{{u +
							v}}{2}}}{{\left( {b - a} \right) - \left( {v - u} \right)}}}
		\right)^2 } \right]
	{\rm{B}}\left({2,1+\frac{1}{q}}\right)\cdot\left\Vert {g^{\prime
	}}\right\Vert _{p,[a,b]} ,\,\,\,\,\,\,\,\,\,g^{\prime }\in
	L_{p}\left[ {a,b}\right],
	\end{array}%
	\right.  \nonumber
	\end{align}%
	where, $p>1$ and $\frac{1}{p}+\frac{1}{q}=1$.
\end{thm}

\begin{proof}
	Using the fact that for a Riemann integrable function
	$p:[c,d]\rightarrow
	\mathbb{R}$ and $L$-Lipschitzian function $\nu :[c,d]\rightarrow \mathbb{R}$%
	, one has the inequality
	\begin{align}
	\left\vert {\int_{c}^{d}{p\left( t\right) d\nu \left( t\right)
	}}\right\vert \leq L\int_{c}^{d}{\left\vert {p\left( t\right)
		}\right\vert dt}. \label{eq4.5.27}
	\end{align}%
	As $f$ is $L$--Lipschitzian on $[a,b]$, by (\ref{eq4.5.27}) we
	have
	\begin{align}
	&\left\vert {\mathcal{T}_a^{v} \left( {f,g} \right) -
		\mathcal{T}_{u}^b \left( {f,g} \right)}\right\vert
	\nonumber\\
	& \leq
	\frac{L}{v-a}\int_{a}^{v}{\left\vert {\left( {r-a}\right) \left[ {\frac{1}{{r-a}}%
				\int_{a}^{r}{g\left( s\right)
					ds}-\frac{1}{{v-a}}\int_{a}^{v}{g\left( s\right) ds}}\right]
		}\right\vert ds}
	\label{eq4.5.28}\\
	&\qquad +\frac{L}{b-u}\int_{u}^{b}{\left\vert {\left( {t-u}\right)
			\left[ {\frac{1}{{t-u}} \int_{u}^{t}{g\left( s\right)
					ds}-\frac{1}{{b-u}}\int_{u}^{b}{g\left(s\right) ds}}\right]
		}\right\vert dt}
	\nonumber\\
	& \leq \frac{1}{2}L\left\Vert {g^{\prime }}\right\Vert _{\infty } \left[{\frac{1}{v-a}\int_{a}^{v}%
		{\left( {r-a}\right) \left( {v-r}\right) dr}+\frac{1}{b-u}\int_{u}^{b}%
		{\left( {t-a}\right) \left( {b-t}\right) dt}}\right]
	\nonumber\\
	&= \frac{1}{6}L\left\Vert {g^{\prime }}\right\Vert _{\infty }
	\left[{  \frac{{\left( {v - a} \right)^2 +\left( {b - u}
				\right)^2}}{2} }\right]\nonumber
	\\
	&=L\frac{{\left[ {\left( {b - a} \right) - \left( {v - u} \right)}
			\right]}}{6} \cdot \left[ {\frac{1}{4} + \left( {\frac{{\frac{{a +
							b}}{2} - \frac{{u + v}}{2}}}{{\left( {b - a} \right) - \left( {v -
						u} \right)}}} \right)^2 } \right] \left\Vert {g^{\prime
	}}\right\Vert _{\infty },\nonumber
	\end{align}%
	where we used the inequality (\ref{Bar4.3.1}), with $%
	d=r,t$ and $c=a,u$;  respectively.
	
	To obtain the second inequality, setting $d=r,t$ and $c=a,u$;
	respectively, in (\ref{Cer4.3.2}),  we get
	\begin{multline}
	\left\vert {\frac{1}{{r-a}}\int_{a}^{r}{g\left( s\right) ds}-\frac{1}{{v-a}%
		}\int_{a}^{v}{g\left( s\right) ds}}\right\vert  \label{eq4.5.29} \\
	\leq \frac{\left( {v-r}\right)^{\frac{1}{q}}}{{\left( {q+1}%
			\right) ^{1/q}\left( {v-a}\right) ^{\frac{1}{q}}}}\left[ {\left(
		{r - a} \right)^q  + \left( {v - r} \right)^q } \right]^{1/q}
	\left\Vert {%
		g^{\prime }}\right\Vert _{p,[a,v]}
	\end{multline}
	and
	\begin{multline}
	\left\vert {\frac{1}{{t-u}}\int_{u}^{t}{g\left( s\right) ds}-\frac{1}{{b-u}%
		}\int_{u}^{b}{g\left( s\right) ds}}\right\vert  \label{eq4.5.30}
	\\
	\leq \frac{\left( {b-t}\right)^{\frac{1}{q}}}{{\left( {q+1}%
			\right) ^{1/q}\left( {b-u}\right) ^{\frac{1}{q}}}}\left[ {\left(
		{t - u} \right)^q  + \left( {b - t} \right)^q } \right]^{1/q}
	\left\Vert {%
		g^{\prime }}\right\Vert _{p,[u,b]}
	\end{multline}
	Substituting (\ref{eq4.5.29}) and (\ref{eq4.5.30}) in
	(\ref{eq4.5.28}), we get
	\begin{align*}
	&\left\vert {\mathcal{T}_a^{v} \left( {f,g} \right) -
		\mathcal{T}_{u}^b \left( {f,g} \right)}\right\vert
	\\
	&\le L\frac{{\left\Vert {g^{\prime }}\right\Vert
			_{p,[a,v]}}}{{\left( {q+1}\right)^{1/q} \left( {v - a} \right)^{1
				+ \frac{1}{q}}  }}\int_{a}^{v}{\left( {r-a}\right) \left(
		{b-r}\right)^{\frac{1}{q}} \left[ {\left( {r - a} \right)^q +
			\left( {v - r} \right)^q } \right]^{1/q} dr}
	\\
	&\qquad+L\frac{{\left\Vert {g^{\prime }}\right\Vert
			_{p,[u,b]}}}{{\left( {q+1}\right)^{1/q}\left( {b - u} \right)^{1 +
				\frac{1}{q}} }}\int_{u}^{b}{\left( {t-u}\right) \left(
		{b-t}\right)^{\frac{1}{q}} \left[ {\left( {t - u} \right)^q  +
			\left( {b - t} \right)^q } \right]^{1/q} dt}
	\\
	& \leq L\frac{{\left\Vert {g^{\prime }}\right\Vert
			_{p,[a,v]}}}{{\left( {q+1}\right) ^{1/q}\left( {v - a}
			\right)^{1 + \frac{1}{q}} }}\mathop {\sup }\limits_{r\in \left[ {a,v}%
		\right] }  \left[ {\left( {r - a} \right)^q  + \left( {v - r}
		\right)^q } \right]^{1/q} \int_{a}^{v}{\left( {r-a}\right) \left(
		{v-r}\right) ^{\frac{1}{q}}dr}
	\\
	&\qquad+L\frac{{\left\Vert {g^{\prime }}\right\Vert
			_{p,[u,b]}}}{{\left( {q+1}\right) ^{1/q}\left( {b - u}
			\right)^{1 + \frac{1}{q}} }}\mathop {\sup }\limits_{t\in \left[ {u,b}%
		\right] }  \left[ {\left( {t - u} \right)^q  + \left( {b - t}
		\right)^q } \right]^{1/q} \int_{u}^{b}{\left( {t-u}\right) \left(
		{b-t}\right) ^{\frac{1}{q}}dt}
\\
	&=L\frac{{\left( {v-a}\right) ^{2}}}{{\left( {q+1}\right)
			^{1/q}}}{\rm{B}}\left({2,1+\frac{1}{q}}\right)\cdot \left\Vert
	{g^{\prime }}\right\Vert _{p,[a,v]} +L\frac{{\left( {b-u}\right)
			^{2}}}{{\left( {q+1}\right)
			^{1/q}}}{\rm{B}}\left({2,1+\frac{1}{q}}\right)\cdot \left\Vert
	{g^{\prime }}\right\Vert _{p,[u,b]}
	\\
	&\le L\frac{{\left\Vert {g^{\prime }}\right\Vert
			_{p,[a,b]}}}{{\left( {q+1}\right)
			^{1/q}}}{\rm{B}}\left({2,1+\frac{1}{q}}\right)\cdot
	\left[{\left( {v-a}\right) ^{2}
		+\left( {b-u}\right) ^{2} }\right]
	\\
	&=\frac{{2\left[ {\left( {b - a} \right) - \left( {v - u} \right)}
			\right]}}{{\left( {q+1}\right) ^{1/q}}}   \cdot \left[
	{\frac{1}{4} + \left( {\frac{{\frac{{a + b}}{2} - \frac{{u +
							v}}{2}}}{{\left( {b - a} \right) - \left( {v - u} \right)}}}
		\right)^2 } \right]
	{\rm{B}}\left({2,1+\frac{1}{q}}\right)\cdot\left\Vert {g^{\prime
	}}\right\Vert _{p,[a,b]}
	\end{align*}%
	
	which proves the second   inequality in (\ref{eq4.4.26}).
\end{proof}

\begin{cor}
	Under the assumptions of  Theorem \ref{thm4.5.7}, then
	\begin{align}
	&\left\vert {\mathcal{T}_a^{u} \left( {f,g} \right) -
		\mathcal{T}_{u}^b \left( {f,g} \right)}\right\vert  \label{eq4.5.31} \\
	& \leq L\left\{
	\begin{array}{l}
	\frac{1}{6}\left\Vert {g^{\prime }}\right\Vert _{\infty } \left[{
		\frac{{\left( {u - a} \right)^2 +\left( {b - u}
				\right)^2}}{2} }\right],\,\,\,\,\,\,\,\,\,\,\,\,\,\,\,\,g^{\prime }\in L_{\infty }\left[ {%
		a,b}\right] ; \\
	\\
	\frac{{ \left[{\left( {u-a}\right) ^{2} +\left( {b-u}\right) ^{2}
			}\right]}}{{\left( {q+1}\right)
			^{1/q}}}{\rm{B}}\left({2,1+\frac{1}{q}}\right)\cdot\left\Vert
	{g^{\prime }}\right\Vert _{p,[a,b]} ,\,\,\,\,\,\,\,\,\,g^{\prime
	}\in L_{p}\left[ {a,b}\right],
	\end{array}%
	\right.  \nonumber
	\end{align}%
	where, $p>1$ and $\frac{1}{p}+\frac{1}{q}=1$. In particular case,
	if $u=\frac{a+b}{2}$ then
	\begin{align}
	&\left\vert {\mathcal{T}_a^{\frac{a+b}{2}} \left( {f,g} \right) -
		\mathcal{T}_{\frac{a+b}{2}}^b \left( {f,g} \right)}\right\vert  \label{eq4.5.32} \\
	& \leq L\left\{
	\begin{array}{l}
	\frac{\left({b-a}\right)^2}{24}\left\Vert {g^{\prime }}\right\Vert _{\infty } ,\,\,\,\,\,\,\,\,\,\,\,\,\,\,\,\,g^{\prime }\in L_{\infty }\left[ {%
		a,b}\right] ; \\
	\\
	\frac{{  \left( {b-a}\right) ^{2}  }}{{2\left( {q+1}\right)
			^{1/q}}}{\rm{B}}\left({2,1+\frac{1}{q}}\right)\cdot\left\Vert
	{g^{\prime }}\right\Vert _{p,[a,b]} ,\,\,\,\,\,\,\,\,\,g^{\prime
	}\in L_{p}\left[ {a,b}\right],
	\end{array}%
	\right.  \nonumber
	\end{align}%
\end{cor}

\begin{thm}
	\label{thm4.5.9} Let $f,g:[a,b]\rightarrow \mathbb{R}$ be such
	that $f$ is $L$--Lipschitzian on $[a,b]$ and $g$ is of
	$p$-$H$--H\"{o}lder type on $[a,b]$ where $p\in(0,1]$ and $H>0$
	are given, then
	\begin{align}
	\left\vert {\mathcal{T}_a^{v} \left( {f,g} \right) -
		\mathcal{T}_{u}^b \left( {f,g} \right)}\right\vert\label{eq4.5.33}
	\leq \frac{L H}{(p+1)^2(p+2)}  \cdot
	\left[{\frac{(b-a)+(v-u)}{2}+\left|{\frac{u+v}{2}-\frac{a+b}{2}}\right|}\right]^{p+1}.
	\end{align}
\end{thm}

\begin{proof}
	We repeat the proof of Theorem \ref{thm4.5.7}. As $f$ is
	$L$--Lipschitzian and $g$ is of $p$-$H$--H\"{o}lder type on
	$[a,b]$, by (\ref{Cer4.3.3}) we have
	\begin{align*}
	&\left\vert {\mathcal{T}_a^{v} \left( {f,g} \right) -
		\mathcal{T}_{u}^b \left( {f,g} \right)}\right\vert
	\\
	& \leq \frac{L}{v-a}\int_{a}^{v}{\left\vert {\left( {r-a}%
			\right) \left[ {\frac{1}{{r-a}}\int_{a}^{r}{g\left( s\right) ds}-\frac{1}{{%
						v-a}}\int_{a}^{v}{g\left( s\right) ds}}\right] }\right\vert dr}
	\\
	&\qquad +\frac{L}{b-u}\int_{u}^{b}{\left\vert {\left( {t-u}%
			\right) \left[ {\frac{1}{{t-u}}\int_{u}^{t}{g\left( u\right) du}-\frac{1}{{%
						b-u}}\int_{u}^{b}{g\left( u\right) du}}\right] }\right\vert dt}
	\\
	& \leq \frac{LH}{(p+1)(v-a)}\int_{a}^{v}{\left( {r-a}\right)
		\left( {v-r}\right)^{p}dr}
	+\frac{LH}{(p+1)(b-u)}\int_{u}^{b}{\left( {t-u}\right) \left(
		{b-t}\right)^{p}dt}
	\\
	& =\frac{LH(v-a)^{p+1}}{(p+1)^2(p+2)}
	+\frac{LH(b-u)^{p+1}}{(p+1)^2(p+2)}
\\
	&\le \frac{LH}{(p+1)} {\rm{B}}\left({p+1,2}\right) \cdot
	\left[\max\{{(v-a),(b-u)}\}\right]^{p+1}
	\\
	&= \frac{LH}{(p+1)^2(p+2)}  \cdot
	\left[{\frac{(b-a)+(v-u)}{2}+\left|{\frac{u+v}{2}-\frac{a+b}{2}}\right|}\right]^{p+1},
	\end{align*}%
	where  for the last inequality  a simple calculation yields that
	\begin{align*}
	\int_{a}^{b}{\left( {t-a}\right) \left( {b-t}\right) ^{p}dt}=\left( {b-a}%
	\right) ^{p+2}\int_{0}^{1}{\left( {1-t}\right) t^{p}dt}=\frac{\left( {b-a}%
		\right) ^{p+2}}{{\left( {p+1}\right) \left( {p+2}\right) }},
	\end{align*}
	which completes the proof.
\end{proof}

\begin{cor}
	\label{cor4.5.10} Let $f,g$ be two Lipschitzian mappings on
	$[a,b]$ with Lipschitz constants $L_f, L_g>0$, then
	\begin{align}
	\left\vert {\mathcal{T}_a^{v} \left( {f,g} \right) -
		\mathcal{T}_{u}^b \left( {f,g} \right)}\right\vert\label{eq4.5.34}
	\leq \frac{L_fL_g}{12}  \cdot
	\left[{\frac{(b-a)+(v-u)}{2}+\left|{\frac{u+v}{2}-\frac{a+b}{2}}\right|}\right]^{2}.
	\end{align}
	Moreover,
	\begin{align}
	\left\vert {\mathcal{T}_a^{u} \left( {f,g} \right) -
		\mathcal{T}_{u}^b \left( {f,g} \right)}\right\vert
	\label{eq4.5.35} \le \frac{L_fL_g}{12}  \cdot
	\left[{\frac{b-a}{2}+\left|{u-\frac{a+b}{2}}\right|}\right]^{2},
	\end{align}
	for all $a\le u \le b$. In particular case if $u=\frac{a+b}{2}$,
	we have
	\begin{align}
	& \left\vert {\mathcal{T}_a^{\frac{a+b}{2}} \left( {f,g} \right) -
		\mathcal{T}_{\frac{a+b}{2}}^b \left( {f,g} \right)}\right\vert
	\label{eq4.5.36}\le \frac{1}{24} L_f L_g\left({b-a}\right)^2.
	\end{align}
\end{cor}
\begin{proof}
	In (\ref{eq4.5.33}), let $p=1$ we get  (\ref{eq4.5.34}). The
	inequality (\ref{eq4.5.35}) can be obtained by setting
	$v=u+\epsilon$, $\epsilon>0$, and letting $\epsilon\to 0^+$.
\end{proof}

\begin{thm} \label{thm4.5.12}Let $f,g:[a,b]\rightarrow \mathbb{R}$ be two
	absolutely continuous on $[a,b]$. If $f' \in L_{\alpha}[a,b]$,
	$\alpha,\beta>1$, $\frac{1}{\alpha}+\frac{1}{\beta}=1$, then
	\begin{multline}
	\left\vert {\mathcal{T}_a^{v} \left( {f,g} \right) -
		\mathcal{T}_{u}^b \left( {f,g} \right)}\right\vert\\\\
	\leq \left\{
	\begin{array}{l}
	\frac{ (v-a)^{\frac{1}{\beta}}+(b-u)^{\frac{1}{\beta}} }{2}\cdot
	{\rm B}^{\frac{1}{\beta}} \left( {\beta+1,\beta+1}
	\right)\cdot\left\|f'\right\|_{\alpha,[a,b]}\cdot
	\left\|g'\right\|_{\infty,[a,b]},\\\qquad\qquad\qquad\qquad\qquad\qquad\qquad\qquad\qquad\qquad{\rm{if}}\,\,\,\,g^{\prime }\in L_{\infty}\left[ {a,b}%
	\right]\\
	\\
	\frac{\left( {v - a} \right)^{1  + \frac{1}{\beta}}+\left( {b - u}
		\right)^{1  + \frac{1}{\beta}}
	}{{\left( {q+1}%
			\right)^{1/q}}}\cdot{\rm B}^{\frac{1}{\beta}}\left( {\beta
		+1,\frac{\beta}{q} +1} \right)  \left\Vert { g^{\prime
	}}\right\Vert _{p,[a,b]}\left\|f'\right\|_{\alpha,[a,b]};
	\\
	\qquad\qquad\qquad\qquad\qquad\qquad\qquad\qquad\qquad{\rm{if}}\,\,\,\,g^{\prime }\in L_{p}\left[ {a,b}%
	\right]\\
	\,\,\,\,\,\,\,\,\,\,\,\,\,\,\,\,\,\,\,\,\,\,\,\,\,\,\,\,\,\,\,\,\,\,\,\,\,\,\,\,\,\,
	\,\,\,\,\,\,\,\,\,\,\,\,\,\,\,\,\,\,\,\,\,\,\,\,\,\,\,\,\,\,\,\,\,\,\,\,\,\,\,\,\,\,
	\,\,\,\,\,\,\,\,\,\,\,\,\,\,\,\,\,\,\, p>1 ,\,\,\frac{1}{p}+\frac{1}{q}=1, \\
	\\
	\left[{\left( {v - a} \right)^{1 + \frac{1}{\beta}}+ \left( {b -
			u} \right)^{1 + \frac{1}{\beta}} }\right] \cdot {\rm
		B}^{\frac{1}{\beta}}\left( {\beta +1,\beta +1}
	\right)\cdot\left\Vert {g^{\prime }}\right\Vert
	_{1,[a,b]}\left\|f'\right\|_{\alpha,[a,b]},\\\qquad\qquad\qquad\qquad\qquad\qquad\qquad\qquad\qquad\qquad\qquad{\rm{if}}\,\,\,\,g^{\prime }\in L_{1}\left[ {a,b}%
	\right]
	\end{array}%
	\right.\label{eq4.5.37}
	\end{multline}%
\end{thm}

\begin{proof}
	Taking the absolute value in (\ref{eq4.2.14}) and utilizing the
	triangle inequality. As $f^{\prime } \in L_{\alpha}([a,b])$, by
	H\"{o}lder inequality we have
	\begin{align}
	&\left\vert {\mathcal{T}_a^{v} \left( {f,g} \right) -
		\mathcal{T}_{u}^b \left( {f,g} \right)}\right\vert
	\label{eq4.5.38}\\
	&\leq \frac{1}{v-a}\int_{a}^{v}{\left\vert {\left( {r-a}\right) \left[ {\frac{1}{{r-a}}%
				\int_{a}^{r}{g\left( s\right)
					ds}-\frac{1}{{v-a}}\int_{a}^{v}{g\left( s\right) ds}}\right]
		}\right\vert \left\vert f'\left( {r}\right) \right\vert dr}
	\nonumber\\
	&\qquad+\frac{1}{b-u}\int_{u}^{b}{\left\vert {\left( {t-u}\right) \left[ {\frac{1}{{t-u}}%
				\int_{u}^{t}{g\left( s\right)
					ds}-\frac{1}{{b-u}}\int_{u}^{b}{g\left( s\right) ds}}\right]
		}\right\vert \left\vert f'\left( {t}\right) \right\vert dt}
	\nonumber\\
	&\leq \frac{1}{v-a}\left(\int_{a}^{v}{\left| {r-a}\right|^{\beta} \left| {\frac{1}{{r-a}}%
			\int_{a}^{r}{g\left( s\right)
				ds}-\frac{1}{{v-a}}\int_{a}^{v}{g\left( s\right)
				ds}}\right|^{\beta} dr}\right)^{1/{\beta}}
	\nonumber\\
	&\qquad\qquad\qquad\qquad\times\left(\int_{a}^{v}{\left\vert
		f'\left( {r}\right) \right\vert^{\alpha}
		dr}\right)^{1/{\alpha}}\nonumber
	\\
	&\qquad+\frac{1}{b-u}\left(\int_{a}^{b}{\left| {t-u}\right|^{\beta} \left| {\frac{1}{{t-u}}%
			\int_{u}^{t}{g\left( s\right)
				ds}-\frac{1}{{b-u}}\int_{u}^{b}{g\left( s\right)
				ds}}\right|^{\beta} dt}\right)^{1/{\beta}}
	\nonumber\\
	&\qquad\qquad\qquad\qquad\times\left(\int_{u}^{b}{\left\vert
		f'\left( {t}\right) \right\vert^{\alpha}
		dt}\right)^{1/{\alpha}}\nonumber
	\end{align}
	Now, in (\ref{Bar4.3.1}) put $d=r,t$ and $c=a,u$; respectively,
	then
	\begin{align*}
	\left| {\frac{1}{{r-a}}%
		\int_{a}^{r}{g\left( s\right)
			ds}-\frac{1}{{v-a}}\int_{a}^{v}{g\left( s\right) ds}}\right| \le
	\frac{v-r}{2(v-a)}\cdot \left\|g'\right\|_{\infty,[a,v]}
	\end{align*}
	and
	\begin{align*}
	\left| {\frac{1}{{t-u}}%
		\int_{u}^{t}{g\left( s\right)
			du}-\frac{1}{{b-u}}\int_{u}^{b}{g\left(s\right) ds}}\right| \le
	\frac{b-t}{2(b-u)}\cdot \left\|g'\right\|_{\infty,[u,b]}
	\end{align*}
	Substituting these inequalities in (\ref{eq4.5.38}) we get
	\begin{align*}
	&\left\vert {\mathcal{T}_a^{v} \left( {f,g} \right) -
		\mathcal{T}_{u}^b \left( {f,g} \right)}\right\vert
	\\
	&\leq \frac{1}{2(v-a)^2}
	\cdot\left\|f'\right\|_{\alpha,[a,v]}\cdot
	\left\|g'\right\|_{\infty,[a,v]} \left(\int_a^b{ \left(
		{r-a}\right)^{\beta} \left( {v-r}\right)^{\beta}dr}
	\right)^{\frac{1}{\beta}}
	\\
	&\qquad+ \frac{1}{2(b-u)^2}
	\cdot\left\|f'\right\|_{\alpha,[u,b]}\cdot
	\left\|g'\right\|_{\infty,[u,b]} \left(\int_a^b{ \left(
		{t-u}\right)^{\beta} \left( {b-t}\right)^{\beta}dt}
	\right)^{\frac{1}{\beta}}
	\\
	&=\frac{(v-a)^{\frac{1}{\beta}}}{2}\cdot {\rm B}^{\frac{1}{\beta}}
	\left( {\beta+1,\beta+1}
	\right)\cdot\left\|f'\right\|_{\alpha,[a,v]}\cdot
	\left\|g'\right\|_{\infty,[a,v]}
	\\
	&\qquad + \frac{(b-u)^{\frac{1}{\beta}}}{2}\cdot {\rm
		B}^{\frac{1}{\beta}} \left( {\beta+1,\beta+1}
	\right)\cdot\left\|f'\right\|_{\alpha,[u,b]}\cdot
	\left\|g'\right\|_{\infty,[u,b]}
	\\
	&\le \frac{ (v-a)^{\frac{1}{\beta}}+(b-u)^{\frac{1}{\beta}}
	}{2}\cdot {\rm B}^{\frac{1}{\beta}} \left( {\beta+1,\beta+1}
	\right)\cdot\left\|f'\right\|_{\alpha,[a,b]}\cdot
	\left\|g'\right\|_{\infty,[a,b]}
	\end{align*}
	which prove the first inequality in (\ref{eq4.5.37}).
	
	To prove the second and third inequalities in (\ref{eq4.5.37}), we
	apply (\ref{Cer4.3.2}) by setting $d=r,t$ and $c=a,u$;
	respectively, then we get
	\begin{align}
	&\int_{a}^{v}{\left| {r-a}\right|^{\beta} \left| {\frac{1}{{r-a}}%
			\int_{a}^{r}{g\left(s\right)
				ds}-\frac{1}{{v-a}}\int_{a}^{v}{g\left( s\right)
				ds}}\right|^{\beta}
		dr}   \nonumber\\
	&\leq \left\{
	\begin{array}{l}
	\frac{\left\Vert { g^{\prime }}\right\Vert^\beta _{p,[a,v]}
	}{{\left( {q+1}%
			\right)^{\beta/q}\left( {v-a}\right) ^{{\frac{\beta}{q}}}}}
	\int_a^v{\left( {r-a}\right)^{\beta}\left( {v-r}\right)
		^{\frac{\beta}{q}}  \left[ {\left( {r - a} \right)^q  + \left( {v
				- r} \right)^q } \right]^{\beta/q}dr},\\
	\qquad\qquad\qquad\qquad\qquad\qquad\qquad\qquad\text{if}\,\,\,g^{\prime
	}\in L_{p}\left[ {a,v}\right]
	;\,\,\, \\\qquad\qquad \qquad\qquad\qquad\qquad\qquad\qquad p>1 ,\,\,\frac{1}{p}+\frac{1}{q}=1, \\
	\nonumber\\
	\frac{1}{{(v-a)^{\beta}}}\left\Vert {g^{\prime
	}}\right\Vert^{\beta}
	_{1,[a,v]}\cdot \int_a^v{(r-a)^{\beta}(v-r)^{\beta}dr},\,\,\,\,\,\text{if}\,\,\,\,g^{\prime }\in L_{1}\left[ {a,v}%
	\right] .%
	\end{array}%
	\right.
	\nonumber\\
	&\le\left\{
	\begin{array}{l}
	\frac{\left\Vert { g^{\prime }}\right\Vert^\beta _{p,[a,v]}
	}{{\left( {q+1}%
			\right)^{\beta/q}\left( {v-a}\right) ^{{ \frac{\beta}{q}} }}}
	\mathop {\sup }\limits_{r\in \left[ {a,v}%
		\right] }  \left[ {\left( {r - a} \right)^q  + \left( {v - r}
		\right)^q } \right]^{\beta/q} \int_a^v{\left(
		{r-a}\right)^{\beta}\left( {v-r}\right)
		^{\textstyle{\frac{\beta}{q}}}dr},\\
	\qquad\qquad\qquad\qquad\qquad\qquad\qquad\qquad\text{if}\,\,\,g^{\prime
	}\in L_{p}\left[ {a,v}\right]
	;\,\,\, \\\qquad\qquad \qquad\qquad\qquad\qquad\qquad\qquad p>1 ,\,\,\frac{1}{p}+\frac{1}{q}=1, \\
	\\\\
	\frac{1}{{(v-a)^{\beta}}}\left\Vert {g^{\prime
	}}\right\Vert^{\beta} _{1,[a,v]}\cdot  \left( {v - a}
	\right)^{2\beta + 1} {\rm B}\left( {\beta+1 ,\beta +1} \right)
	,\,\,\,\,\,\text{if}\,\,\,\,g^{\prime }\in L_{1}\left[ {a,v}%
	\right] .%
	\end{array}%
	\right.
	\nonumber\\ \nonumber\\
	&=\left\{
	\begin{array}{l}
	\frac{ \left( {v - a} \right)^{\left( {2 + \frac{1}{q}}
			\right)\beta  + 1}
	}{{\left( {q+1}%
			\right)^{\beta/q}\left( {v-a}\right)
			^{{\frac{\beta}{q}}}}}\cdot{\rm B}\left( {\beta+1 ,\frac{\beta}{q}
		+1} \right)  \left\Vert { g^{\prime }}\right\Vert^\beta
	_{p,[a,v]},
	\,\,\,\,\,\,\,\,\,\,\,\,\text{if}\,\,\,g^{\prime
	}\in L_{p}\left[ {a,v}\right]
	;\,\,\, \\\qquad\qquad \qquad\qquad\qquad\qquad\qquad\qquad p>1 ,\,\,\frac{1}{p}+\frac{1}{q}=1,  \\
	\\
	\frac{1}{{(v-a)^{\beta}}}\left\Vert {g^{\prime
	}}\right\Vert^{\beta} _{1,[a,v]}\cdot  \left( {v - a}
	\right)^{2\beta +
		1} {\rm B}\left( {\beta+1 ,\beta +1} \right),\,\,\,\,\,\text{if}\,\,\,\,g^{\prime }\in L_{1}\left[ {a,v}%
	\right] . \label{eq4.5.39}
	\end{array}%
	\right..
	\end{align}
	Similarly, we have
	\begin{align}
	&\int_{u}^{b}{\left| {t-u}\right|^{\beta} \left| {\frac{1}{{t-u}}%
			\int_{u}^{t}{g\left(s\right)
				ds}-\frac{1}{{b-u}}\int_{u}^{b}{g\left( s\right)
				ds}}\right|^{\beta} dt}
	\nonumber\\
	&=\left\{
	\begin{array}{l}
	\frac{ \left( {b - u} \right)^{\left( {2 + \frac{1}{q}}
			\right)\beta  + 1}
	}{{\left( {q+1}%
			\right)^{\beta/q}\left( {b-u}\right)
			^{{\frac{\beta}{q}}}}}\cdot{\rm B}\left( {\beta+1 ,\frac{\beta}{q}
		+1} \right)  \left\Vert { g^{\prime }}\right\Vert^\beta
	_{p,[u,b]},
	\,\,\,\,\,\,\,\,\,\,\,\,\text{if}\,\,\,g^{\prime
	}\in L_{p}\left[ {u,b}\right]
	;\,\,\, \\\qquad\qquad \qquad\qquad\qquad\qquad\qquad\qquad\qquad p>1 ,\,\,\frac{1}{p}+\frac{1}{q}=1,  \\
	\\
	\frac{1}{{(b-u)^{\beta}}}\left\Vert {g^{\prime
	}}\right\Vert^{\beta} _{1,[u,b]}\cdot  \left( {b - u}
	\right)^{2\beta +
		1} {\rm B}\left( {\beta+1 ,\beta +1} \right),\,\,\,\,\,\text{if}\,\,\,\,g^{\prime }\in L_{1}\left[ {u,b}%
	\right] .%
	\end{array}%
	\right..\label{eq4.5.40}
	\end{align}
	Substituting (\ref{eq4.5.39}) and (\ref{eq4.5.40}) in
	(\ref{eq4.5.38}), we get
	\begin{align*}
	& \left\vert {\mathcal{T}_a^{v} \left( {f,g} \right) -
		\mathcal{T}_{u}^b \left( {f,g} \right)}\right\vert\\
	& \leq  \left\{
	\begin{array}{l}
	\frac{\left( {v - a} \right)^{1  + \frac{1}{\beta}}+\left( {b - u}
		\right)^{1  + \frac{1}{\beta}}
	}{{\left( {q+1}%
			\right)^{1/q}}}\cdot{\rm B}^{\frac{1}{\beta}}\left( {\beta
		+1,\frac{\beta}{q} +1} \right)  \left\Vert { g^{\prime
	}}\right\Vert _{p,[a,b]}\left\|f'\right\|_{\alpha,[a,b]}
	\\\\
	\left[{\left( {v - a} \right)^{1 + \frac{1}{\beta}}+ \left( {b -
			u} \right)^{1 + \frac{1}{\beta}} }\right] \cdot {\rm
		B}^{\frac{1}{\beta}}\left( {\beta +1,\beta +1}
	\right)\cdot\left\Vert {g^{\prime }}\right\Vert
	_{1,[a,b]}\left\|f'\right\|_{\alpha,[a,b]}
	\end{array}
	\right.
	\end{align*}
	for all $p,q,\alpha,\beta>1$ with $\frac{1}{p}+\frac{1}{q}=1$ and
	$\frac{1}{\alpha}+\frac{1}{\beta}=1$, which proves the second and
	the third inequalities in (\ref{eq4.5.37}).
\end{proof}

\begin{cor} \label{cor4.5.13} Under the assumptions of Theorem
	\ref{thm4.5.12}, we have
	\begin{multline}
	\left\vert {\mathcal{T}_a^{u} \left( {f,g} \right) -
		\mathcal{T}_{u}^b \left( {f,g} \right)}\right\vert \\\\
	\leq \left\{
	\begin{array}{l}
	\frac{ (u-a)^{\frac{1}{\beta}}+(b-u)^{\frac{1}{\beta}} }{2}\cdot
	{\rm B}^{\frac{1}{\beta}} \left( {\beta+1,\beta+1}
	\right)\cdot\left\|f'\right\|_{\alpha,[a,b]}\cdot
	\left\|g'\right\|_{\infty,[a,b]},\\\qquad\qquad\qquad\qquad\qquad\qquad\qquad\qquad\qquad\qquad{\rm{if}}\,\,\,\,g^{\prime }\in L_{\infty}\left[ {a,b}%
	\right]\\
	\\
	\frac{\left( {u - a} \right)^{1  + \frac{1}{\beta}}+\left( {b - u}
		\right)^{1  + \frac{1}{\beta}}
	}{{\left( {q+1}%
			\right)^{1/q}}}\cdot{\rm B}^{\frac{1}{\beta}}\left( {\beta
		+1,\frac{\beta}{q} +1} \right)  \left\Vert { g^{\prime
	}}\right\Vert _{p,[a,b]}\left\|f'\right\|_{\alpha,[a,b]};
	\\
	\qquad\qquad\qquad\qquad\qquad\qquad\qquad\qquad\qquad{\rm{if}}\,\,\,\,g^{\prime }\in L_{p}\left[ {a,b}%
	\right]\\
	\,\,\,\,\,\,\,\,\,\,\,\,\,\,\,\,\,\,\,\,\,\,\,\,\,\,\,\,\,\,\,\,\,\,\,\,\,\,\,\,\,\,
	\,\,\,\,\,\,\,\,\,\,\,\,\,\,\,\,\,\,\,\,\,\,\,\,\,\,\,\,\,\,\,\,\,\,\,\,\,\,\,\,\,\,
	\,\,\,\,\,\,\,\,\,\,\,\,\,\,\,\,\,\,\, p>1 ,\,\,\frac{1}{p}+\frac{1}{q}=1, \\
	\\
	\left[{\left( {u - a} \right)^{1 + \frac{1}{\beta}}+ \left( {b -
			u} \right)^{1 + \frac{1}{\beta}} }\right] \cdot {\rm
		B}^{\frac{1}{\beta}}\left( {\beta +1,\beta +1}
	\right)\cdot\left\Vert {g^{\prime }}\right\Vert
	_{1,[a,b]}\left\|f'\right\|_{\alpha,[a,b]},\\\qquad\qquad\qquad\qquad\qquad\qquad\qquad\qquad\qquad\qquad\qquad{\rm{if}}\,\,\,\,g^{\prime }\in L_{1}\left[ {a,b}%
	\right]
	\end{array}%
	\right..\label{eq4.5.41}
	\end{multline}%
	In particular case, if $u=\frac{a+b}{2}$ we get
	\begin{multline}
	\left\vert {\mathcal{T}_a^{\frac{a+b}{2}} \left( {f,g} \right) -
		\mathcal{T}_{\frac{a+b}{2}}^b \left( {f,g} \right)}\right\vert \\\\
	\leq \left\{
	\begin{array}{l}
	\left( {\frac{{b - a}}{2}} \right)^{\frac{1}{\beta }} \cdot {\rm
		B}^{\frac{1}{\beta}} \left( {\beta+1,\beta+1}
	\right)\cdot\left\|f'\right\|_{\alpha,[a,b]}\cdot
	\left\|g'\right\|_{\infty,[a,b]},\,\,{\rm{if}}\,\,\,\,g^{\prime }\in L_{\infty}\left[ {a,b}%
	\right]\\
	\\
	\frac{{\left( {b - a} \right)^{1 + \frac{1}{\beta }} }}{{2^{1 +
				\frac{1}{\beta }} \left( {q + 1} \right)^{\frac{1}{q}} }}
	\cdot{\rm B}^{\frac{1}{\beta}}\left( {\beta +1,\frac{\beta}{q} +1}
	\right)  \left\Vert { g^{\prime }}\right\Vert
	_{p,[a,b]}\left\|f'\right\|_{\alpha,[a,b]};\,\,
	{\rm{if}}\,\,\,\,g^{\prime }\in L_{p}\left[ {a,b}%
	\right]\\
	\,\,\,\,\,\,\,\,\,\,\,\,\,\,\,\,\,\,\,\,\,\,\,\,\,\,\,\,\,\,\,\,\,\,\,\,\,\,\,\,\,\,
	\,\,\,\,\,\,\,\,\,\,\,\,\,\,\,\,\,\,\,\,\,\,\,\,\,\,\,\,\,\,\,\,\,\,\,\,\,\,\,\,\,\,
	\,\,\,\,\,\,\,\,\,\,\,\,\,\,\,\,\,\,\, p>1 ,\,\,\frac{1}{p}+\frac{1}{q}=1, \\
	\\
	\left( {\frac{{b - a}}{2}} \right)^{1 + \frac{1}{\beta }} \cdot
	{\rm B}^{\frac{1}{\beta}}\left( {\beta +1,\beta +1}
	\right)\cdot\left\Vert {g^{\prime }}\right\Vert
	_{1,[a,b]}\left\|f'\right\|_{\alpha,[a,b]},\,{\rm{if}}\,\,\,\,g^{\prime }\in L_{1}\left[ {a,b}%
	\right]
	\end{array}%
	\right..\label{eq4.5.41}
	\end{multline}%
\end{cor}

\begin{rem}
	For the second inequality in (\ref{eq4.5.37}) we have the
	following particular cases:
	\begin{enumerate}
		\item If $\alpha=p$ and $\beta=q$, then we have
		\begin{align}
		 \left\vert {\mathcal{T}_a^{v} \left( {f,g} \right) -
			\mathcal{T}_{u}^b \left( {f,g} \right)}\right\vert
		 \le\frac{\left( {v - a} \right)^{1  + \frac{1}{q}}+\left( {b - u}
			\right)^{1  + \frac{1}{q}}
		}{{\left( {q+1}%
				\right)^{1/q}}}\cdot{\rm B}^{\frac{1}{q}}\left( {q +1,2} \right)
		\left\Vert { g^{\prime }}\right\Vert
		_{p,[a,b]}\left\|f'\right\|_{p,[a,b]}.
		\end{align}
		\noindent Therefore, as $v\to u^+$ we have
		\begin{align}
		 \left\vert {\mathcal{T}_a^{u} \left( {f,g} \right) -
			\mathcal{T}_{u}^b \left( {f,g} \right)}\right\vert
		 \le\frac{\left( {u - a} \right)^{1  + \frac{1}{q}}+\left( {b - u}
			\right)^{1  + \frac{1}{q}}
		}{{\left( {q+1}%
				\right)^{1/q}}}\cdot{\rm B}^{\frac{1}{q}}\left( {q +1,2} \right)
		\left\Vert { g^{\prime }}\right\Vert
		_{p,[a,b]}\left\|f'\right\|_{p,[a,b]},
		\end{align}
		and for $u=\frac{a+b}{2}$ we have
		\begin{align}
		 \left\vert {\mathcal{T}_a^{\frac{a+b}{2}} \left( {f,g} \right) -
			\mathcal{T}_{\frac{a+b}{2}}^b \left( {f,g} \right)}\right\vert
		 \le\frac{{\left( {b - a} \right)^{1 + \frac{1}{q}} }}{{2^{1 +
					\frac{1}{q }} \left( {q + 1} \right)^{\frac{1}{q}} }}\cdot{\rm
			B}^{\frac{1}{q}}\left( {q +1,2} \right) \left\Vert { g^{\prime
		}}\right\Vert _{p,[a,b]}\left\|f'\right\|_{p,[a,b]}.
		\end{align}
		
		\item If $\alpha=q$ and $\beta=p$, then we have
		\begin{align}
		 \left\vert {\mathcal{T}_a^{v} \left( {f,g} \right) -
			\mathcal{T}_{u}^b \left( {f,g} \right)}\right\vert
		 \le\frac{\left( {v - a} \right)^{1  + \frac{1}{p}}+\left( {b - u}
			\right)^{1  + \frac{1}{p}}
		}{{\left( {q+1}%
				\right)^{1/q}}}\cdot{\rm B}^{\frac{1}{p}}\left( {p +1,\frac{p}{q}
			+1} \right)  \left\Vert { g^{\prime }}\right\Vert
		_{p,[a,b]}\left\|f'\right\|_{q,[a,b]}.
		\end{align}
		Similarly, as $v\to u^+$, we have
		\begin{align}
		\left\vert {\mathcal{T}_a^{u} \left( {f,g} \right) -
			\mathcal{T}_{u}^b \left( {f,g} \right)}\right\vert
		\le\frac{\left( {u - a} \right)^{1  + \frac{1}{p}}+\left( {b - u}
			\right)^{1  + \frac{1}{p}}
		}{{\left( {q+1}%
				\right)^{1/q}}}\cdot{\rm B}^{\frac{1}{p}}\left( {p +1,\frac{p}{q}
			+1} \right)  \left\Vert { g^{\prime }}\right\Vert
		_{p,[a,b]}\left\|f'\right\|_{q,[a,b]},
		\end{align}
		and for $u=\frac{a+b}{2}$ we have
		\begin{align}
		\left\vert {\mathcal{T}_a^{\frac{a+b}{2}} \left( {f,g} \right) -
			\mathcal{T}_{\frac{a+b}{2}}^b \left( {f,g} \right)}\right\vert
		\le\frac{{\left( {b - a} \right)^{1 + \frac{1}{p}} }}{{2^{1 +
					\frac{1}{p}} \left( {q + 1} \right)^{\frac{1}{q}} }}\cdot{\rm
			B}^{\frac{1}{p}}\left( {p +1,\frac{p}{q} +1} \right)  \left\Vert {
			g^{\prime }}\right\Vert _{p,[a,b]}\left\|f'\right\|_{q,[a,b]},
		\end{align}
	\end{enumerate}
	for all $p,q>1$ with $\frac{1}{p}+\frac{1}{q}$.
\end{rem}

\begin{rem}
	In this work, all obtained bounds for the difference between two
	\v{C}eby\v{s}ev functional were taken under the assumption that
	$\left[a,v\right] \cap \left[u,b\right]= \left[u,v\right]$. The
	same bounds hold with a few changes in the case that
	$\left[u,v\right] \subset \left[a,b\right]$. Namely, replace every
	`$a$' (in the obtained results) by `$u$'; every `$u$' (in the
	obtained results) by `$a$' and accordingly the differences
	$(v-u),(b-a)$ instead of $(v-a),(b-u)$.
\end{rem}

\begin{rem}
	All obtained bounds hold for the \v{C}eby\v{s}ev functional
	$\left\vert {\mathcal{T}_a^{b} \left( {f,g} \right)}\right\vert$,
	this can be done by noting that  $\left\vert {\mathcal{T}_a^{v}
		\left( {f,g} \right) - \mathcal{T}_{u}^b \left( {f,g}
		\right)}\right\vert \longrightarrow \left\vert {\mathcal{T}_a^{b}
		\left( {f,g} \right)}\right\vert$ as $v=u \longrightarrow a $.
\end{rem}

\end{document}